\documentclass{amsart}
\usepackage[utf8x]{inputenc}
\usepackage{amsmath, amsthm, amssymb}
\usepackage[usenames,dvipsnames,svgnames,table]{xcolor}
\usepackage[margin=1.3in]{geometry}
\usepackage{enumerate}
\usepackage{dsfont}
\usepackage{cleveref}
\usepackage{cite}
\usepackage{mathtools}

\newtheorem{theorem}{Theorem}[section]
\newtheorem{proposition}[theorem]{Proposition}
\newtheorem{lemma}[theorem]{Lemma}

\newtheorem*{remark}{Remark}

\newcommand{\R}{\mathbb R}
\newcommand{\T}{\mathbb T}

\newcommand{\dd}{\, \mathrm{d}}

\newcommand{\vv}{\langle v\rangle}
\newcommand{\vvp}{\langle v'\rangle}
\newcommand{\vvO}{\langle v_0\rangle}

\newcommand{\les}{\lesssim}

\DeclareMathOperator{\Id}{Id}

\def\comma{ {\rm ,\qquad{}} }

\newcommand{\be}{\begin{equation}}
\newcommand{\ee}{\end{equation}}

\newcommand{\ns}{{\rm ns}}
\newcommand{\s}{{\rm s}}

\numberwithin{equation}{section}

\title[Local well-posedness for the Boltzmann equation]{Local well-posedness for the Boltzmann equation with very soft potential and 
polynomially decaying initial data}
\author{Christopher Henderson}
\address{Department of Mathematics, University of Arizona, Tucson, AZ 85721}
\email{ckhenderson@math.arizona.edu}
\author{Weinan Wang}
\address{Department of Mathematics, University of Arizona, Tucson, AZ 85721}
\email{weinanwang@math.arizona.edu}

\begin{document}

\begin{abstract}
	In this paper, we address the local well-posedness of the spatially inhomogeneous non-cutoff Boltzmann equation when the initial data decays polynomially in the velocity variable. We consider the case of very soft potentials $\gamma + 2s < 0$.  Our main result completes the picture for local well-posedness in this decay class by removing the restriction $\gamma + 2s > -3/2$ of previous works.  
	Our approach is entirely based on the Carleman decomposition of the collision operator into a lower order term and an integro-differential operator similar to the fractional Laplacian.  Interestingly, this yields a very short proof of local well-posedness when $\gamma \in (-3,0]$ and $s \in (0,1/2)$ in a weighted $C^1$ space that we include as well.
\end{abstract}

\maketitle

\section{Introduction}

The Boltzmann equation is a kinetic equation arising in statistical physics. 
Its solution $f(t,x,v)\geq 0$ models the density of particles of a diffuse gas at time $t\in [0,T]$, at location $x\in \T^3$, and with velocity $v\in \R^3$.  Roughly, each particle travels with a fixed velocity until a collision at which time it takes on a new velocity chosen in a way compatible with physical laws. In this article, we focus on the non-cutoff version of \eqref{e:Boltzmann} that includes the physically realistic singularity at grazing collisions.  The equation reads
\begin{equation}\label{e:Boltzmann}
	\begin{cases}
		\partial_t f + v\cdot \nabla_x f
			= Q(f,f)
			\qquad &\text{ in } [0,T] \times \T^3 \times \R^3,\\
		f(0,\cdot,\cdot)
			= f_{\rm in} \geq 0
			\qquad &\text{ in } \T^3\times \R^3.
	\end{cases}
\end{equation}
The collision operator $Q$ is defined by
\[
	Q(f,f)
		= \int_{\R^3}\int_{{\mathbb S}^2} B(v-v_*,\sigma)
			\left( f(v_*')f(v') - f(v_*)f(v) \right) \dd\sigma \dd v_*,
\]
where $v$ and $v_*$ are pre-collisional velocities and $v'$ and $v_*'$ are post-collisional velocities related by
\[
v' = \frac{v+v_*}{2} + \sigma \frac{|v-v_*|}{2}
	\qquad\text{ and }\qquad
v_*' = \frac{v+v_*}{2} - \sigma \frac{|v-v_*|}{2}
\]
and the collision kernel $B$ is given by
\[
	B(v-v_*,\sigma) = |v-v_*|^\gamma \theta^{-2-2s} \tilde b(\cos \theta),
		\qquad \text{ where }
	\cos \theta = \sigma \cdot \frac{v-v_*}{|v-v_*|},~
	\gamma \in (-3,1],~
	s\in(0,1),
\]
and $\tilde b$ is a positive bounded function.  
In this work, we are mostly interested in the regime of very soft potentials, that is, when $\gamma + 2s < 0$.  

There are several active research directions regarding the well-posedness of the Boltzmann equation: global well-posedness in the spatially homogeneous setting (that is, $x$-independent), global well-posedness and regularity of weak solutions, global well-posedness and convergence of close-to-equilibrium solutions, and local well-posedness with large initial data.  Here, we are interested in the local well-posedness of~\eqref{e:Boltzmann} with large initial data, and, as such, leave it to other references to detail the extensive history of research into the first three categories (see, e.g., \cite{amuxy2011hardpotentials, amuxy2011global, alexandre2002longrange, alexandre2004landau, alonso2018non, chaturvedi2020stability, chen2011boltzmann, desvillettes2005boltzmann, desvillettes2009homogeneous, desvillettes2004smooth, diperna1989cauchy, goudon, gressman2011boltzmann, he2017global, herau2019regularization, luk2019vacuum, silvestre2021solutions, villani1998homogeneous}).

Alexandre, Morimoto, Ukai, Xu, and Yang, often referred to by the acronym AMUXY, made the first serious progress on the local well-posedness theory for the (non-cutoff) Boltzmann equation.  In particular, in a sequence of seminal works, by deriving new estimates on the collision operator $Q$, they were able to establish local well-posedness under the condition that $e^{\alpha |v|^2} f_{\rm in}$ is bounded in certain Sobolev-based spaces~\cite{amuxy2010regularizing, amuxy2011qualitative, amuxy2011bounded, amuxy2011uniqueness, amuxy2013mild}.  We note that the Gaussian decay plays a large role in their analysis to compensate for moment loss.

The first results weakening the Gaussian-decay condition on the initial data are due to Morimoto and Yang \cite{morimoto2015polynomial}.   They established local well-posedness in an $H^6$-based space under the assumptions that $\gamma \in (-3/2,0]$ and $s \in (0,1/2)$.   This was later extended by Henderson, Snelson, and Tarfulea~\cite{HST_boltz_wp}, who showed local well-posedness in an $H^5$-based space under the assumption $s\in (0,1)$ and $\max{\{-3, -\frac{3}{2}-2s\}}<\gamma<0$.  Our goal, in the present work, is to remove the restriction $\gamma + 2s > -3/2$.   In general, the larger $\gamma + 2s$ is, the more the decay of $f$ at $|v|= +\infty$ is the issue, and the smaller (more negative) $\gamma + 2s$ is, the more regularity is the issue. 

Our interest in establishing local well-posedness with initial data that is merely polynomially decaying is due to its relationship to the recent conditional regularity program initiated by Silvestre~\cite{silvestre2016boltzmann} and continued in collaboration with Imbert and Mouhot~\cite{imbert2018decay, IMS_Gaussian, imbert2018schauder, imbert2019smooth, imbert2020regularity, IS2020weakharnack}.  The goal of the program is to understand the regularity theory for the Boltzmann equation conditional to the mass, energy, and entropy densities
\[
	M(t,x)
		= \int f(t,x,v) dv,
		\quad
	E(t,x)
		= \int f(t,x,v) |v|^2 dv
	\quad\text{and}\quad
	H(t,x)
		= \int f(t,x,v) \log f(t,x,v) dv,
\]
satisfying, uniformly in $(t,x)$,
\be\label{e.macroscopic_bounds}
	M, E, H \leq C
		\quad\text{and}\quad
	\frac{1}{C} \leq M
		\qquad \text{ for all } (t,x),
\ee
where $C$ is a positive constant. When $f$ is $x$-independent, it is well-known that these conditions are always satisfied.  To date, Imbert, Mouhot, and Silvestre have developed a Harnack inequality and Schauder estimates, obtained a sharp lower bound on the tail behavior of $f$, and proved a propagation of polynomial upper bounds of $f$ result, all of which depended only on the bounds in~\eqref{e.macroscopic_bounds}.


An upshot of the program of Imbert, Mouhot, and Silvestre is that, roughly, when a suitable local well-posedness result exists, solutions may be continued as long as~\eqref{e.macroscopic_bounds} holds (see~\cite[Section 1.1.2]{imbert2019smooth} and~\cite[Corollary 1.2]{HST_boltz_wp}).  In particular, the local well-posedness result must allow for polynomially decaying initial data as that type of decay can be propagated forward in time depending only on the constant $C$ in~\eqref{e.macroscopic_bounds}.  As no such propagation-of-decay result exists for Gaussian decay, the classical results of AMUXY cannot be used.  It is for this reason that it is important to develop the local well-posedness theory when $f_{\rm in}$ decays only polynomially.


Our main theorem removes the restriction of previous results~\cite{HST_boltz_wp,morimoto2015polynomial} that $\gamma + 2s > -3/2$, thereby completing the picture for local well-posedness with polynomially decaying initial data when $\gamma \in (-3,0)$ and $s\in (0,1)$.  In order to state our result, we define the following two spaces: given $k, n, m\geq 0$ and $T>0$, let
\be\label{d:X space}
	X^{k,n,m}
		= H^{k,n}(\T^3\times \R^3)\cap L^{\infty,m}(\T^3\times \R^3)
	\qquad\text{and}\qquad
	Y_T^{k,n,m}
		= L^\infty([0,T]; X^{k,n,m}).
\ee
For any $p\geq 1$, we use $L^{p,n}$ to refer to the space of functions $g$ such that $\vv^n g \in L^p$, where $\vv^2 = 1 + |v|^2$.  The weighted Sobolev space $H^{k,n}$ is defined analogously.
\begin{theorem}\label{t.main}
Assume that $\gamma + 2s < 0$, $k\geq 5$, $n>3/2$, and $m>M=M(k,n,\gamma,s)$ sufficiently large. Suppose $0\leq f_{\rm in}\in X^{k,n,m}$. 
Then there exists a time $T>0$, depending only on $\|f_{\rm in}\|_{X^{k,n,m}}$ as well as $n$, $k$, $m$, $\gamma$, $s$, and $\tilde b$, and a unique solution $f \in Y_{T}^{k,n,m}\cap C([0,T];H^{k,n}_{x,v})$ of~\eqref{e:Boltzmann} such that $f\geq 0$ and
\[
	\|f\|_{Y_T^{k,n,m}} \lesssim \|f_{\rm in}\|_{X^{k,n,m}}.
\]
\end{theorem}

As discussed above, an important motivation of \Cref{t.main} is to extend the continuation criterion for the Boltzmann equation to the very soft potentials range. 
%
%
%
While such a result does not directly follow from \Cref{t.main} and the Imbert-Mouhot-Silvestre regularity program (as these results only deal with the regime $\gamma + 2s \in [0,2]$), it is likely a straightforward exercise after adding an additional assumption on the $L^\infty_{t,x}L^p_v$ norm of $f$ to~\eqref{e.macroscopic_bounds}.  Indeed, this has already been accomplished for the closely-related Landau equation in~\cite{HST2018landau} in the analogous parameter regime.  An upshot of such a continuation criterion, were it established, is the ability to construct classical solutions from rough initial initial data as accomplished for the Landau equation~\cite{HST2019rough}. These will be the subject of a future work.

As is typical for nonlinear equations, the main step in the proof of local well-posedness of~\eqref{e:Boltzmann} is to establish {\em a priori} estimates on solutions.  In particular, this requires obtaining bounds on the collision operator $Q$ as a bilinear form from and to various Banach spaces.

In order to explain the strategy and difficulties in obtaining such estimates, we discuss the restriction $\gamma + 2s > -3/2$ in~\cite{HST_boltz_wp}.  This is inherited in the application of the estimates on the collision operator developed in~\cite{amuxy2011qualitative}.  For certain key estimates, AMUXY use Fourier analysis, which is most suited $L^2$-based spaces.  However, a major lesson from~\cite{silvestre2016boltzmann} is that one can, roughly, think of $Q$ as having a coefficient of the form
\[
	\int f(t,x,w) |v-w|^{\gamma + 2s} dw,
\]
and one sees, after applying the Cauchy-Schwarz inequality, that such coefficients are bounded using the (weighted) $L^2$-norm of $f$ only when $\gamma + 2s > -3/2$.


In view of the above, it is required to develop new estimates on the collision operator in spaces that are not $L^2$-based.  Our approach is to take advantage of the Carleman decomposition (see equation \eqref{e.carleman1}), which views $Q$ as the sum of an integro-differential operator similar to the fractional Laplacian and a lower order term.  As this is a real space-based approach, it is possible, through intricate analysis, to obtain estimates on $Q(g,f)$ in various spaces depending on both $L^2$- and $L^\infty$-based norms.  This allows us to circumvent the issues encountered in previous works.

Curiously, this approach makes an extremely simple proof of local well-posedness in a weighted $C^1$ space obvious when $s\in (0,1/2)$.  The reason for this is as follows.  First, as observed in~\cite[Proposition~3.2]{HST_boltz_wp}, the Carleman decomposition makes it easy to obtain $L^{\infty,m}$ bounds on $f$ from $\|f_{\rm in}\|_{L^{\infty,m}}$ via a simple comparison principle argument.  The important observation is that, roughly, at a maximum of $\vv^m f$, the only high order term has a good sign.  A straightforward attempt to repeat this for the $L^{\infty,m}$ norm of $\partial f$ is complicated by the fact that $\partial f$ solves an equation involving a term $Q(\partial f, f)$.  From the Carleman decomposition of $Q$, we, roughly, see
\[
	Q(\partial f, f)
		\sim \left( \int \partial f(w) |v-w|^{\gamma +2s} dw\right)
			\Delta^{s} f
			\lesssim \|\partial f\|_{L^{\infty,m}} \|\vv^m f\|_{C^{2s+\epsilon}}.
\]
Fortunately, when $s \in (0,1/2)$, this is lower order and the previous argument can be repeated.


We now state the result.  We first define the $C^k$ analogue of the spaces $X$ and $Y_T$~\eqref{d:X space}:
\be\label{e.tilde_XY}
	\begin{split}
	&\tilde X^{k, m_0,m_1}
		= \{f: \vv^{m_0} \nabla^{\ell} f \in C(\T^3\times \R^3)^{6^\ell} \text{ for } 0 \leq \ell \leq k-1, \vv^{m_1} \nabla^k f \in C(\T^3\times \R^3)^{6^k}\}
	\\&\text{and}\quad
	\tilde Y_T^{k,m_0,m_1}
		= L^\infty([0,T]; \tilde X^{k,m_0,m_1}).
	\end{split}
\ee

\begin{theorem}\label{t:C1_wellposed}
Let $k\geq 1$, $\gamma \in (-3,0]$, $s\in (0,1/2)$, $m_1 > 3 + \gamma + 2s$, and $m_0$ be sufficiently large depending only on $k$, $\gamma$, $s$, and $m_1$.  Let the initial data $0\leq f_{\rm in}\in \tilde X^{k,m_0,m_1}$.
Then there exists a time $T>0$, depending only on $\|f_{\rm in}\|_{\tilde X^{k,m_0,m_1}}$, $\gamma$, $s$, $m_0$, $m_1$, and $\tilde b$, and a unique solution $f\geq 0$ of~\eqref{e:Boltzmann} such that
\[
	\begin{split}
		&f \in \tilde Y_T^{k,m_0,m_1}
			\cap C([0,T]; C^k(\T^3\times \R^3))
			\cap C^1([0,T]; C^{k-1}(\T^3\times \R^3))
		\\&
		\text{and }
		\|f\|_{\tilde Y_T^{k,m_0,m_1}}
			\lesssim \|f_{\rm in}\|_{\tilde X^{k,m_0,m_1}}.
	\end{split}
\]
\end{theorem}

We note that simply by differentiating the equation, we can obtain further time regularity when $k>1$.  In addition, a careful accounting based on the estimates of the collision operator and the definition of $\tilde Y^{k,m_0,m_1}_T$ yields  the decay in velocity of the $C^\ell_{x,v}$ norms.  This is not the main interest of the statement above so we omit it.

The significance of \Cref{t:C1_wellposed}, besides having such a short proof, is that it improves on \cite{morimoto2015polynomial, HST_boltz_wp} in two major ways.  First, it increases the range of possible $\gamma$: \cite{morimoto2015polynomial} requires $\gamma \in (-3/2,0]$ and~\cite{HST_boltz_wp} requires $\gamma \in (-3/2-2s,0)$.  Second, it weakens conditions on the initial regularity: \cite{morimoto2015polynomial} works in an $H^6$-based space and~\cite{HST_boltz_wp} works in an $H^5$-based space, both of which embed in $C^1$.  Note that it also reduces the regularity required of the initial data in comparison to \Cref{t.main}.  On the other hand, like \cite{morimoto2015polynomial} but unlike~\cite{HST_boltz_wp} and \Cref{t.main}, it only applies to $s \in(0,1/2)$.



\subsection{Notation}\label{sec:notation}

We use the notation $A \lesssim B$ if there is a constant $C$ such that $A \leq C B$.  In general, the constant $C$ may depend on $\gamma$, $s$, $n$, $m$, $k$, $m_0$, $m_1$, and $\tilde b$. Additionally, if an assumption  for an estimate involves a requirement such as $\alpha >\beta$, then the constant $C$ may depend on $\alpha - \beta$. We use $A \approx B$ if $A \lesssim B$ and $B\lesssim A$.  Occasionally, it will be necessary to include a constant, in which case we use $C$ to represent such a constant and this constant $C$ may change line-by-line.

Any integral whose domain of integration in $v$ is not specified is understood to be an integral over $\R^3$ and any integral whose domain of integration in $x$ is not specified is understood to be an integral over $\T^3$.  For example, for any measurable $\varphi$ and any measurable sets $\Omega_x \subset \T^3$ and $\Omega_v\subset \R^3$, we have
\[
	\int \int_{\Omega_v} \varphi(x,v) dv dx
		= \int_{\T^3} \int_{\Omega_v} \varphi(x,v) dv dx
	\quad\text{and}\quad
	\int_{\Omega_x} \int \varphi(x,v) dv dx
		= \int_{\Omega_x} \int_{\R^3} \varphi(x,v) dv dx.
\]
Similarly, we often suppress the domain in Lebesgue, Sobolev, and H\"older spaces when it is clear, writing, e.g., $f \in L^{\infty,m}$ instead of $f \in L^{\infty,m}(\R^3)$ if has already been established that $f: \R^3 \to \R$.

We use $B_R$ to mean a ball of radius $R$ around the origin.  Whenever the ball is not centered at the origin, we denote the center $v_0$ as $B_R(v_0)$.

Finally, when stating estimates on the collision operator $Q(g,f)$, we often omit the assumptions on the involved functions $g$ and $f$.  In these cases, the estimate holds whenever the right hand side is finite.

 \subsection{Outline} The rest of the paper is organized as follows. In \Cref{s:collision}, we consider bounds on the collision operator.  In particular, we recall useful known results, prove some easy extensions of them, and state our main new estimates.  Then, in \Cref{s:lwp}, we prove the existence and uniqueness of solutions using the bounds from \Cref{s:collision}. Afterwards, in \Cref{s:collision_proofs}, we prove the estimates on the collision operator $Q$. Finally, in \Cref{s:simple}, we give a simple proof of local well-posedness for the case of $s\in (0,1/2)$ and $\gamma \in (-3,0]$.

 \section{Estimates on the collision operator}\label{s:collision}

In this section, we state the key estimates on the collision operator $Q$ that we use in our proof of well-posedness.  We begin with a brief overview of the Carleman decomposition allowing us to use ideas from the study of integro-differential operators.  Then we state known estimates and their easy extensions.  Finally, we state new estimates whose proof, contained in \Cref{s:collision_proofs}, makes up the bulk of this manuscript.

 \subsection{Carleman decomposition}
 
 A key tool in our analysis is the Carleman decomposition \cite{carleman1933boltzmann, carleman1957} that views the Boltzmann collision operator $Q$ as the sum of a non-local diffusion operator locally similar to $-(-\Delta)^s$ and a lower order reaction term.  This decomposition is well-known, see \cite{alexandre20003d} for an early discussion of it and \cite[Sections 4 and 5]{silvestre2016boltzmann} for the presentation used here.  Indeed,
 \be\label{e.carleman1}
 	\begin{split}
		&Q(g,f)
			= Q_\s(g,f) + Q_\ns(g,f)\\
		&Q_\s(g,f) = \int (f(v') - f(v)) K_g(v,v') dv'
			\qquad \\
		&Q_\ns(g,f) = c_b (S_\gamma * g) f
		,
	\end{split}
 \ee
 where $S_\gamma(v) = |v|^\gamma$, $c_b>0$ is a fixed constant, and $K_g$ satisfies, for any $g\geq 0$ and any $v,v'\in \R^3$,
 \be\label{e.carleman2}
 	\begin{split}
		&K_g(v,v')
			\approx \frac{1}{|v-v'|^{3+2s}} \int_{w \in v + (v'-v)^\perp} g(w) |v-w|^{\gamma + 2s + 1} dw\\
		&\text{and} \qquad
		K_g(v,v+v') = K_g(v,v-v').
	\end{split}
\ee
We refer to $Q_\s$ as the ``singular'' part and $Q_\ns$ as the ``non-singular'' part.

Actually, to be fully rigorous, $Q_\s$ should be defined using a principal value.  We abuse notation and suppress this as all our estimates occur over symmetric domains near the base point $v$  and are, thus, compatible with the limit involved in the principal value.

\subsection{Previously established estimates and easy extensions}

In this section, we state various estimates on the collision operator that are well-known or are simple extensions of previous results.

%

	\begin{lemma}[Estimates of the kernel $K_{g}$]
		\label{p.w3311}
		For all $r>0$ and $v\in \R^3$,
		\begin{enumerate}[(i)]
			\item 
			\label{l.is3.4}
			 $\displaystyle
			\int\limits_{ B_{2r}(v) \backslash B_{r}(v)} K_g(v',v)\,dv',
			\int\limits_{ B_{2r}(v) \backslash B_{r}(v)} K_g(v,v')\,dv'
			\lesssim
			r^{-2s}
				\int |g(z)||z-v|^{\gamma+2s}\, dz$.\medskip
			\item 
			\label{l.is3.6}
			\quad $\displaystyle
			\left|\,\int [K_g(v,v')-K_g(v',v)]\,dv'\right|
			\lesssim
			\int |g(z)||z-v|^{\gamma}\,dz
			.
			$\medskip
			\item 
			\label{l.is3.7}
			\quad$\displaystyle
			\int\limits_{ B_{r}(v)} (v'-v)K_g(v,v')\,dv'
				= 0					.
			$
			\item 
			\label{l.is3.72}
			\quad$\displaystyle
			\Big|\int_{ B_{r}(v)} (v'-v)K_g(v,v')\,dv\Big|
				\lesssim
				\int |g(z)||z-v'|^{1+\gamma}\,dz
			.$
		\end{enumerate}
	\end{lemma}
\Cref{p.w3311} follows from \cite[Lemmas 3.4, 3.5, 3.6, and 3.7]{IS2020weakharnack}. 
%
%
%
%
%
%
%
%
The following lemma can be regarded as a slight generalization of \cite[Proposition 2.1]{IMS_Gaussian}.
\begin{lemma}\label{l:new kernel}
	For $0<s<1$, $\alpha > 2s$, $r>0$, and $g: \R^3 \to \R_+$ there holds
	\[
	\int_{B_r(v')} K_g(v',v) |v-v'|^\alpha dv,
		\int_{B_r(v')} K_g(v,v') |v-v'|^\alpha dv
	\lesssim r^{\alpha-2s} \int |g(w)| |v-w|^{\gamma + 2s} dw.
	\]
\end{lemma}
\begin{proof}
	The proofs of both inequalities are similar, so we show only the latter.  Assume without loss of generality that $g \geq 0$.  We proceed with a simple annular decomposition paired with the existing estimate \Cref{p.w3311}.\eqref{l.is3.4}. Indeed, letting $A_k = B_{2^{-k}r}(v')\setminus B_{2^{-k-1}r}(v')$, we have
	\[
	\begin{split}
	\int_{B_r(v')} &K_g(v,v') |v-v'|^\alpha dv
	= \sum_{k=0}^\infty \int_{A_k} K_g(v,v') |v-v'|^\alpha dv
	\leq \sum_{k=0}^\infty 2^{-\alpha k} r^\alpha \int_{A_k} K_{g}(v,v') dv\\
	&\leq \sum_{k=0}^\infty 2^{-\alpha k} r^\alpha \int_{B_{2^{-k-1}r}^c(v')} K_{g}(v,v') dv
	\lesssim \sum_{k=0}^\infty 2^{- k ( \alpha -2s)} r^{\alpha-2s} \int g(w) |v'-w|^{\gamma+2s} dw.
	\end{split}
	\]
	The claim then follows due to the fact that the sum over $k$ is finite.
\end{proof}

The next lemma concerns bounds on $K_g$ via $\|g\|_{L^{\infty,m}}$.
\begin{lemma}
	\label{l:2.3}
		Fix any $m>3 + \gamma + 2s$, $g \in L^{\infty, m}$, and $v,v'\in \mathbb R^3$. Then
		\[
			|K_{g}(v,v')|
			\lesssim
			\frac{1}{|v-v'|^{3 + 2s}}\|g\|_{L^{\infty, m}}
			\langle v \rangle^{\gamma+2s+1}.
		\]
\end{lemma}
We omit the proof as this is obvious from~\eqref{e.carleman2} and a straightforward parametrization of the 2-dimensional hyperplane that is the domain of integration.

On the other hand, under a smallness condition on $v'$, we can establish a refined estimate involving the decay of $g$.   To our knowledge this was first observed in \cite[equation~(4.39)]{HST_boltz_wp} but not stated as a stand-alone lemma or given a proof in complete generality.  As such, we include it here.
\begin{lemma}
	\label{l: w0524}
		Fix any $m>3 + \gamma + 2s$, $\theta \in (0,1)$,  $g \in L^{\infty, m}$, and $v,v'\in \mathbb R^3$ with $(1-\theta)|v| \geq |v'|$.  Then
		\[
			|K_g(v,v')|
			\lesssim
			\frac{1}{|v-v'|^{3+2s}}\|g\|_{L^{\infty, m}}
			\langle v \rangle^{\gamma+2s+3 - m}.
		\]
\end{lemma}
As the proof of \Cref{l: w0524} is longer than the others of this subsection, we include it in \Cref{s:collision_proofs}; however, it is simply a more careful writing of the ideas in the proof of \cite[equation~(4.39)]{HST_boltz_wp}.

The next lemma provides estimates for the non-singular part $Q_\ns$. Recall \eqref{e.carleman2}. Then, we have the following estimates.
\begin{lemma}\label{p:non-singular}
	Suppose that $f,g : \T^3 \times \R^3 \to \R$. Then, for any $\epsilon>0$ and $n\geq 0$,
	\[
	\begin{split}
	\| Q_\ns (g,f) \|_{L^{2,n}}
	&\lesssim
	\begin{cases}
	\|g\|_{L^{\infty,3+\gamma+\epsilon}} \|f\|_{L^{2,n}}\\
	\|g\|_{L^{2,n}} \|f\|_{L^{\infty,n + \epsilon + 3/2 +\gamma + (3/2-n)_+}}.
	\end{cases}
	\end{split}
	\]
\end{lemma}
\begin{remark}
	The first inequality in \Cref{p:non-singular} is obvious by writing $Q_\ns(g,f) = (S*g) f$ and bounding $S*g$ in $L^\infty$ using the weighted $L^\infty$ norm of $g$.  The second inequality can be easily proved by using our weighted Young's inequality \Cref{l: Weighted Young}.  As this proof is straightforward from the statement of \Cref{l: Weighted Young}, we omit the details.
\end{remark}

We also require the following from \cite[Lemma 2.6]{HST_boltz_wp}:
\begin{lemma}[Interpolation lemma]\label{interpolation lemma}
	If $n,m \geq 0$, $k' \in (0,k)$, and $l<(m-\frac32)(1-\frac{k'}{k})+n\frac{k'}{k}$, then
	\[
		\|f\|_{H^{k',l}}
	\lesssim
	\|f\|_{L^{\infty, m}}^{1-\frac{k'}{k}}
	\|f\|_{H^{k, n}}^{\frac{k'}{k}}
	.
	\]
\end{lemma}

\subsection{New estimates}

We now state new estimates on the collision operator that are crucial to allowing us to extend well-posedness to the full range of soft potentials.  The prior similar work \cite{HST_boltz_wp} relied heavily on \cite[Theorem 2.4, Proposition 2.5, and Proposition 3.1]{HST_boltz_wp}.  The first two come directly from \cite[Proposition 2.9 and 2.8]{amuxy2011qualitative}, respectively.  Each result, unfortunately, requires $\gamma + 2s > -\frac{3}{2}$.  Thus, these are not applicable in our setting, and the main issue of the present work is to obtain suitable replacements, which we state here.
%
%

The first is a commutator estimate (cf.\ \cite[Proposition~2.8]{amuxy2011qualitative}, \cite[Proposition~2.5]{HST_boltz_wp}). 
\begin{proposition}[Commutator estimate]\label{p:commutator}
	For any $\epsilon>0$, $\gamma \in (-3,0]$, 
	$\mu \in ((1-2s)_+, 2-2s)$, $m > \max\{3+\gamma+2s, \ell + \gamma + \frac{3}{2}\}$, $\ell > \frac{3}{2}$, and $f,g : \R^3\to\R$, we have
		\begin{align*}
		\| \vv^\ell Q_\s(g,f) - Q_\s(g, \vv^\ell f)\|_{L^2}
		\lesssim
			(\|f\|_{L^{2,\ell+3/2+\epsilon}} + \|f\|_{H^{2s -1 + \mu,\mu+\ell+\gamma + 2s}})
				\|g\|_{L^{\infty, m}}.
		\end{align*}
\end{proposition}

The next estimates concern $Q_\s(g,f)$ and involve only the $v$-variable.   
\begin{proposition}\label{p:thm_2.4}
	For any $f,g : \R^3 \to \R$, $\gamma \in (-3,0]$, $\gamma + 2s \leq 0$, and $\epsilon>0$:
	\begin{enumerate}[(i)]
		\item If 
			$\theta \in (0, 2s)$, then
		\[
			\int Q_\s(g,f) h dv
				\lesssim \|g\|_{L^{\infty,3+\gamma+2s+\epsilon}}
					\|f\|_{H^{2s-\theta}}
					\|h\|_{H^{\theta}}.
		\]
		\item If $\theta > 0$, then
		\[
			\|Q_\s(g,f)\|_{L^2}
				\lesssim \|g\|_{L^{\infty,3+\gamma+2s+\epsilon}}
						\|f\|_{H^{2s+\theta}}.
		\]
		\item If $n>3/2+\gamma+2s$, $m>3/2+\gamma+(3/2-n)_{+}$, and $\alpha>2s$, then
			\[
			\|Q_\s(g,f)\|_{L^{2,n}}
				\lesssim
				\|g\|_{L^{2,n}}
				\left(
					\|f\|_{L^{\infty, m}}
					+\|\vv^{n+5/2+ (3/2-n)_+ + \alpha+\gamma+\epsilon}f\|_{C^{\alpha}_{v}}
				\right).
			\]
			\item If $n\geq 0$ and $m > n + 6 + \gamma + 2s$, we have
			\[
			\int \vv^{2n}fQ_\s(g,f)\,dv
			\lesssim
			\|g\|_{L^{\infty,m}}
			\|f\|_{L^{2,n}}^{2} 
			.
			\]
	\end{enumerate}
\end{proposition}
The first two parts above, (i) and (ii), rely heavily on the work in \cite{IS2020weakharnack}; however, that reference is focused on local estimates and, as such, is not concerned with understanding the dependence on weights.  Combined they are a replacement for~\cite[Theorem~2.4]{HST_boltz_wp} (see also \cite[Proposition~2.9]{amuxy2011qualitative}).  The second two parts above, (iii) and (iv), are new.  They are replacements for \cite[Proposition~3.1.(i) and (iii)]{HST_boltz_wp}, respectively.

We make two brief remarks.  First, the result (i) is a slight generalization of the results in~\cite{IS2020weakharnack} as it allows to choose $\theta$ in (i).  Second, the result (ii) almost certainly holds without $\theta=0$; however, as this is not needed for our purposes and the current statement is easy to derive from~\cite{IS2020weakharnack}, we are content to use (ii) as is.

The final estimate makes use of symmetry properties of $Q_\s$ in order to avoid having more than one full derivative ``land'' on $f$.  This is crucial in case two of the proof of the main {\em a priori} estimate \Cref{p:existence}.  It is a replacement for \cite[Proposition~3.1.(iv)]{HST_boltz_wp}.

\begin{proposition}\label{p:symmetry}  
	Suppose that $f,g : \T^3 \times \R^3 \to \R$. If $\gamma \in (-3,0]$, $\epsilon>0$, $\mu \in ((1-2s)_+,2-2s)$, $\kappa\in(s, \min\{2s,1\})$, $n>\frac{3}{2}$, and $m > \max\{3+\gamma+2s, n + \gamma + \frac{3}{2}\}$. Then
	\[
	\begin{split}
	\Big|\int &\vv^{2n} Q_\s (g,f) \partial f dv dx \Big|
	\lesssim
	\|g\|_{L^{\infty, m}}
	\left(
	\|f\|_{L^{2,n + 3/2 + \epsilon}}
	+
	\|f\|_{H^{2s - 1 + \mu,\mu+n +\gamma + 2s}}
	\right)
	\|f\|_{H^{1,n}}
	\\&+
	\|\partial g\|_{H^{3/2 + (2s-1/2)_+ + \epsilon,3+\gamma + 2s + \epsilon}} \|f\|_{H^{1,n}}^2
		+
	\| g\|_{C^{\kappa, 3+\epsilon}}
	\|f\|_{H^{s,n + 3/2 + \epsilon + (\gamma + 2s + 1)_+}}
	\|f\|_{H^{1, n}}
	,
	\end{split}
	\]
	where $\partial = \partial_{x_i}$ or $\partial_{v_i}$ for some $i \in \{1,2,3\}$.
\end{proposition}
Recall that we prove the above estimates in \Cref{s:collision_proofs}.

\section{Existence and uniqueness of solutions: \Cref{t.main}}\label{s:lwp}  
	In this section, we prove \Cref{t.main}.   The majority of the work is in the proof of existence and our approach for this follows \cite{HST_boltz_wp} closely.  Indeed, the construction procedure is similar, relying on exhibiting a solution to a suitably regularized and linearized problem.  We then use compactness to deregularize and a fixed point argument to pass from the linearized problem to the nonlinear one.  The main novelty to the current work as compared to~\cite{HST_boltz_wp} is in the establishment of {\em a priori} estimates in $Y_T^{k,n,m}$ of the regularized and linearized problem.  When possible, we omit details that are unchanged from~\cite{HST_boltz_wp}.

	
\subsection{Proof of existence in \Cref{t.main}}
	First, we define a smooth cut-off function $\psi : \mathbb R^3 \rightarrow \mathbb R$ with $0\leq \psi \leq 1 $, $\int \psi(v)\,dv=1$, 
	\begin{align*}
	\psi = 1,
	\quad\text{on}
	~
	B_{1/2}
	\qquad
	\text{and}
	\qquad
	\psi = 0
	\quad
	\text{on}
	~
	B_{1}^c
	.
	\end{align*}
	Next, for any $\phi: \mathbb T^{3} \times \mathbb R^3 \rightarrow \mathbb R$ and $\epsilon >0$, we define
	\begin{align*}
	\phi^{\epsilon} (x,v)
	=
	\frac{1}{\epsilon^6} \int \psi \Big(\frac{x-y}{\epsilon}\Big)
			\psi \Big(\frac{v-w}{\epsilon}\Big) \phi(y,w)\,dydw
	.
	\end{align*}
	Then we define the regularized collision operator, for any $\delta>0$ and $(x,v) \in \mathbb T^{3} \times \mathbb R^{3}$,
	\begin{align*}
	Q_{\epsilon, \delta} (g(x,\cdot), f(x, \cdot))(v)
	=
	\psi(\delta v)Q(g^{\epsilon} (x,\cdot), \psi(\delta \cdot)f(x,\cdot))
	.
	\end{align*}
	Finally, for any $\sigma \in [0,1]$, we define the differential operator
	\begin{align}\label{e.w451}
	\mathcal{L}_{\sigma, \epsilon, \delta} (f)
	=
	\partial_{t}f
	+
	\sigma \psi(\delta v)v \cdot \nabla_{x}f
	-(\epsilon +(1-\sigma))\Delta_{x,v}f
	-\sigma Q_{\epsilon, \delta} (g,f)
	.
	\end{align}
	The intuition for the above regularizations and cut-offs is given in~\cite[Section 3]{HST_boltz_wp}.
	
	We now establish a priori estimates that hold for both the full equation and the regularized one
	above. This is done in the following proposition.
		\begin{proposition}\label{p:existence}
			Suppose that $T>0, k\geq 5, n>3/2+(\gamma+2s)_{+}, \sigma \in [0,1], \epsilon, \delta\geq0$, and $m\geq 0$. Suppose that $R,f \in Y_{T}^{k,n,m}$
			\begin{equation}\label{e.w316}
			\begin{cases}
			\mathcal L_{\sigma, \epsilon, \delta}f 
			=R
			, \qquad &\text{in}~(0,T)\times \mathbb{T}^{3}\times \mathbb{R}^{3}
			\\f(0,\cdot, \cdot)
			=f_{\text{in}}, \qquad &\text{in}~ \mathbb{T}^{3}\times \mathbb{R}^{3}.
			\end{cases}
			\end{equation}
			 For any $\mu>0$, if $\delta=0$ and $m\geq 3/2+\mu$ or if $\delta>0$, then
			\begin{equation}\label{e:w05081}
			\begin{split}
			\|f\|_{L^{\infty, m}}
			&\leq
			\exp\Big\{C\int_{0}^{T} \|g(t)\|_{L^{\infty, \max\{m,3/2+\mu\}}}dt\Big\}
			\Big(\|f_{\rm in}\|_{L^{\infty, m}}
			+
			\int_{0}^{T} \|R(t)\|_{L^{\infty, m}}dt
			\Big)
			.
			\end{split}
			\end{equation}
			If $\delta=0$ and $m$ is sufficiently large depending on $k,n,\gamma$, and $s$, then
			\begin{equation}\label{e:w05082}
			\begin{split}
			\|f\|_{L^{\infty, m}([0,T];H_{x,v}^{k,n})}
			&\leq
			\exp\Big\{C\int_{0}^{T} \|g(t)\|_{X_{x,v}^{k,n,m}}\,dt\Big\}
			\Big((1+T)\|f_{\rm in}\|_{X_{x,v}^{k,n,m}}
			+
			\int_{0}^{T} \|R(t)\|_{X_{x,v}^{k,n,m}}dt
			\Big)
			.
			\end{split}
			\end{equation}
		\end{proposition}
	Now, we prove \Cref{p:existence}.  The proof follows that of \cite[Proposition~3.1]{HST_boltz_wp} with small changes due to the new estimates on the collision operator necessary in our setting.

	\begin{proof}
	The argument of \eqref{e:w05081} goes exactly as that in \cite[Proposition 3.1]{HST_boltz_wp} and hence we omit the proof here. Now we focus on proving \eqref{e:w05082}. First,  we let $\alpha, \beta \in \mathbb{N}_{0}^3$ be any multi-indices such that $|\alpha|+|\beta|=k$. Then, differentiating \cref{e.w451}, multiplying the resulting equation by $\vv^{2n}\partial_{x}^{\alpha}\partial_{v}^{\beta}f$, integrating in $x$ and $v$, we get
	\[
	\begin{split}
	\frac{1}{2}\frac{d}{dt}\int |\vv^{n}\partial_{x}^{\alpha}\partial_{v}^{\beta}f|^{2}\,dxdv
	&=
	-\sigma \int \left(\sum_{i=1}^{3}\beta_{i}\partial_{x_i}\partial_{x}^{\alpha}\partial_{v}^{\beta-e_{i}}f \right)\vv^{2n}\partial_{x}^{\alpha}\partial_{v}^{\beta}f\,dxdv
	\\&\quad+
	\sigma \sum_{\substack{\alpha'+\alpha''=\alpha \\ \beta'+\beta''=\beta}} C_{\alpha',\beta',\alpha'',\beta''} \int Q(\partial_{x}^{\alpha'}\partial_{v}^{\beta'}g,\partial_{x}^{\alpha'}\partial_{v}^{\beta'}f)\vv^{2n}\partial_{x}^{\alpha}\partial_{v}^{\beta}f\,dxdv
	\\&\quad-
	(\epsilon + 1-\sigma)\int |\nabla_{x,v}\partial_{x}^{\alpha}\partial_{v}^{\beta}f|^{2}
	+
	\int \partial_{x}^{\alpha}\partial_{v}^{\beta}R\vv^{2n}\partial_{x}^{\alpha}\partial_{v}^{\beta}f\,dxdv
	\\&=
	I_{1}+I_{2}+I_{3}+I_{4},
	\end{split}
	\]
	for some constants $C_{\alpha',\beta',\alpha'',\beta''}>0$ depending only on the subscripted quantities.
	
	We see that $I_{1}$ is bounded by $\|f\|_{H^{k,n}}^2$, $I_{4}$ is bounded by  $
	\|R\|_{H^{k,n}}^{2}+\|f\|_{H^{k,n}}^{2}$, and $I_{3}$ has a good sign. Thus, our focus is primarily on $I_{2}$, the term involving the collision operator $Q$. We argue case by case depending on the size of $|\alpha''| + |\beta''|$ in order to establish that	\be\label{e.w05091}
	\begin{split}
	I_{2}
	\lesssim			
	\|g\|_{X^{k,n,m}}
	\|f\|_{X^{k,n,m}}^{2}
	.
	\end{split}
	\ee
	The proof of~\eqref{e.w05091} is postponed momentarily while we show how to conclude.  Indeed,
%
%
	assuming \eqref{e.w05091} is proved, we arrive at
	\begin{equation}
	\begin{split}
		\frac{1}{2}\frac{d}{dt}\int |\vv^{n}\partial_{x}^{\alpha}\partial_{v}^{\beta}f|^{2}\,dxdv
		&\lesssim
		(\|g\|_{X^{k,n,m}}+1)\|f\|_{X^{k,n,m}}^{2}
		+
		\|R\|_{H^{k,n}}^{2}
		.
	\end{split}
	\end{equation}
	Recalling the definition of $X^{k,n,m}$ in~\eqref{d:X space} and using~\eqref{e:w05081}, we find
	\begin{equation}
	\begin{split}
	\frac{1}{2}\frac{d}{dt}\|f\|_{H^{k,n}}^{2}
	&\lesssim
		(\|g(t)\|_{X^{k,n,m}}+1)
			\|f\|_{H^{k,n}}^{2}\\
		&\quad +(\|g(t)\|_{X^{k,n,m}}+1)
			\exp\left\{C\int_0^T\|g(t)\|_{X^{k,n,m}} dt\right\}
			\left(\|f_{\rm in}\|_{L^{\infty, m}}^2
			+ \|R\|_{X^{k,n,m}}^2\right)
		.
	\end{split}
	\end{equation}
	Therefore, we conclude the proof of equation \eqref{e:w05082} by applying Gr\"onwall inequality. 
	
	We now establish \eqref{e.w05091}.	For notational ease, we set
	\begin{align}
	F=\partial^{\alpha''}_{x}\partial^{\beta''}_{v}f
	\comma
	G=\partial^{\alpha'}_{x}\partial^{\beta'}_{v}g
	.
	\end{align}
	Thus, we are estimating terms of the form
	\be\label{e.Weinans_decomposition}
		\begin{split}
		\int \vv^{2n} &Q(G,F) \partial^{\alpha'}_x \partial_v^{\beta'} F dv dx\\
			&= \int \vv^{2n} Q_\s(G,F) \partial^{\alpha'}_x \partial_v^{\beta'} F dv dx
				+ \int \vv^{2n} Q_\ns(G,F) \partial^{\alpha'}_x \partial_v^{\beta'} F dv dx.
		\end{split}
	\ee

	\smallskip
	
	\textbf{Case one:  $|\alpha''|+|\beta''|= k$, i.e., $\alpha''=\alpha$, $\beta''=\beta$, and in the form of $\int \vv^{2n}Q(g,F)F$.}\\
	We estimate the $Q_\s$ term first. We proceed by using \Cref{p:thm_2.4}.(iv), up to increasing $m$ if necessary,
	\[
	\begin{split}
	\int \vv^{2n}Q_\s(g,F)F\,dvdx
	&\lesssim
	\int
	\|F\|_{L_{v}^{2,n}}^{2} 
	\|g\|_{L_{v}^{\infty,m}}
	\,dx
	\lesssim
	\|F\|_{L^{2,n}}^{2} 
	\|g\|_{L^{\infty,m}}\\
	&\lesssim
	\|f\|_{H^{k,n}}^{2}
	\|g\|_{L^{\infty,m}}
	\lesssim
	\|f\|_{X^{k,n,m}}^{2}
	\|g\|_{X^{k,n,m}}
	,
	\end{split}
	\]
	as desired.

	Furthermore, for $\int \vv^{2n}Q_\ns(g,F)F\,dvdx$, we recall~\eqref{e.carleman1} and apply \Cref{p:non-singular} to find 
	\[
	\begin{split}
	\int \vv^{2n}Q_\ns(g,F)F\,dvdx
	&\approx
	\int \vv^{2n} (S_\gamma *g) F^2 dv dx
	\lesssim
		\int \|g\|_{L_{v}^{\infty,n}} \|F\|_{L_{v}^{2,n}}^2\,dx
	\\&\lesssim
		\|g\|_{L^{\infty,m}} \|F\|_{L^{2,n}}
	\lesssim
		\|g\|_{L^{\infty,m}} \|f\|_{H^{k,n}}^2
	\leq
		\|g\|_{X^{k,n,m}}
		\|f\|_{X^{k,n,m}}^2
	.
	\end{split}
	\]
	This concludes the proof of \eqref{e.w05091} in case one.

	\smallskip
	
	\textbf{Case two:  $|\alpha''|+|\beta''|= k-1$, and in the form, $\int \vv^{2n}Q(\partial g,F)\partial F$.} Here we denote derivative operator $\partial=\partial_x^{\alpha'}\partial_v^{\beta'}$ as  $|\alpha'|+|\beta'|=1$.  
	
	We first estimate the $Q_\s$ portion.  Fix $\epsilon \in (0,\min\{s,1-s\})$.  Let $\mu = (1-2s)_+ + \epsilon$, $\kappa = s + \epsilon$, and $\tilde m = \epsilon + \max\{3+\gamma + 2s, n + \gamma + 3/2\}$.  We then directly apply \Cref{p:symmetry} to find
	\[
	\begin{split}
		\left|\int \vv^{2n} Q_\s (\partial g,F) \partial F dv dx \right|
		&\lesssim
			\|\partial g\|_{L^{\infty, \tilde m}}
			\|F\|_{H^{2s -1 + \mu,\mu+n + 3/2}}
			\|F\|_{H^{1,n}}\\
			&\quad+ 
			\|\partial^{2} g\|_{H^{3-s,3+\epsilon}} \|F\|_{H^{1,n}}^2
			+
			\|\partial g\|_{C^{\kappa, 3+\epsilon}}
			\|F\|_{H^{s, n + 5/2 + \epsilon}}
			\|F\|_{H^{1, n}}\\
			&\lesssim 
			\|\partial g\|_{L^{\infty, \tilde m}}
			\|f\|_{H^{k- 2(1-s) + \mu,\mu+n + 3/2}}
			\|f\|_{H^{k,n}}\\
			&\quad+ 
			\|\partial^{2} g\|_{H^{3-s,3+\epsilon}} \|f\|_{H^{k,n}}^2
			+
			\|\partial g\|_{C^{\kappa, 3+\epsilon}}
			\|f\|_{H^{k-(1-s), n + 5/2 + \epsilon}}
			\|f\|_{H^{k, n}}
			.
		\end{split}
	\]
	Applying the Sobolev embedding theorem on terms involving $g$ and then  \Cref{interpolation lemma} (up to increasing $m$ if necessary) yields
	\[
	\begin{split}
	\left|\int \vv^{2n} Q_\s (\partial g,F) \partial F dv dx \right|
	&\lesssim
	\| g\|_{H^{4+\epsilon, \tilde m}}
	\|f\|_{H^{k- (2-2s) + \mu,\mu+n + 3/2}}
	\|f\|_{H^{k,n}}
	+
	\| g\|_{H^{5-s,3+\epsilon}} \|f\|_{H^{k,n}}^2
	\\&\quad
	+
	\| g\|_{H^{4+\kappa, 3+\epsilon}}
	\|f\|_{H^{k-(1-s),n + 5/2 + \epsilon}}
	\|f\|_{H^{k, n}}
	\lesssim
	\|g\|_{X^{k,n,m}}
	\|f\|_{X^{k,n,m}}^{2}
	.
	\end{split}
	\]
	The estimate of the non-singular part $Q_{\ns}$ is the same as in the previous case and is thus omitted.

	\smallskip

	\textbf{Case three:  $|\alpha''|+|\beta''|= k-2$ and $|\alpha'|+|\beta'|=2$.}
	First, we estimate the $Q_\s$ term.  We see
	\be\label{e.c571}
	\begin{split}
	\int \vv^{2n} &Q_\s(G,F)\partial^\alpha_x \partial^\beta_v f\,dvdx
		\leq \|Q_\s(G,F)\|_{L^{2,n}} \|f\|_{H^{k,n}}\\
		&\leq \left(\|Q_\s(G,\vv^n F) - \vv^n Q_\s(G,F)\|_{L^2} + \|Q_\s(G,\vv^n F)\|_{L^2} \right)\|f\|_{H^{k,n}}\\
		&=: (B_1 + B_2)\|f\|_{H^{k,n}}.
	\end{split}
	\ee
	

	We estimate $B_2$ first.  Fix any $\theta \in (0,\min\{\frac{2-2s}{3}, \frac{3}{4}\})$ and let $p = 3/(4\theta)$ and $q = \frac{3}{3-4\theta}$.  
	Then we apply \Cref{p:thm_2.4}.(ii), H\"older's inequality,  and the Sobolev embedding theorem to find
	\[
	\begin{split}
	B_{2}
	&\lesssim
	\Big(
	\int \|G\|_{L^{\infty,3}_{v}}^{2}
	\|F\|_{H^{2s+\theta,n}_{v}}^{2}
	\,dx\Big)^{1/2}
	\lesssim
	\|G\|_{L_{x}^{2p}L^{\infty,3}_{v}}
	\|F\|_{L_{x}^{2q}H^{2s+\theta,n}_{v}}
	\\&\lesssim
	\|G\|_{H_{x}^{3/2-2\theta}H^{3/2+\theta,3}_{v}}
	\|F\|_{H_{x}^{2\theta}H^{2s+\theta,n}_{v}}
	\lesssim
	\|G\|_{H^{3-\theta,3}}
	\|F\|_{H^{2s+3\theta,n}}
	\lesssim
	\|g\|_{X^{k,n,m}}
	\|f\|_{X^{k,n,m}}
	.
	\end{split}
	\]
	The last inequality follows by our choice of $\theta$.
	
	For $B_{1}$, for any $\mu\in ((1-2s)_+, 2-2s)$ and $\tilde m = 1 + \max\{3, n+ \gamma + 3/2\}$, we appeal to our commutator estimate \Cref{p:commutator}, the Cauchy-Schwarz inequality, and the Sobolev embedding theorem to obtain: 
	\[
	\begin{split}
	&B_{1}
	\lesssim
		\left(
		\int (\|F\|_{L_{v}^{2,n+2}} + \|F\|_{H_{v}^{2s -1 + \mu,\mu+n}})^{2}
		\|G\|_{L_{v}^{\infty, \tilde m}}^{2}\,dx\right)^{1/2}
	\\&\lesssim
		\|G\|_{L_{x}^{4}L_{v}^{\infty, \tilde m}} \left(\|F\|_{L_{x}^{4}L_{v}^{2,n+2}}
		+
		\|F\|_{L_{x}^{4}H_{v}^{2s -1 + \mu,\mu+n}}\right)
	\lesssim
		\|G\|_{H^{5/2, \tilde m}} \left(\|F\|_{H^{3/4, n+2}}
		+
		\|F\|_{H_{v}^{2s - 1/4 + \mu,\mu+n}}\right)
	\\&\lesssim
		\|g\|_{H^{9/2, \tilde m}}
		(\|f\|_{H^{k- 5/4,n+2}}
		+
		\|f\|_{H^{k +  2s - 9/4 + \mu,\mu+n}}
		)
	.
	\end{split}
	\]
	Notice that $2s - 9/4 + \mu < 0$ as $\mu < 2-2s$.  With this, observe that all three norms above involve regularity of order strictly less than $k$.  Hence, assuming $m$ is sufficiently large, the interpolation lemma \Cref{interpolation lemma} yields
	\[
	\begin{split}
	B_{1}
 	\lesssim
	\|g\|_{X^{k,n,m}}
	\|f\|_{X^{k,n,m}}
	.
	\end{split}
	\]
	This concludes the estimates for the singular part.
	
	For the non-singular part, we apply \Cref{p:non-singular}  to find
	\[
	\begin{split}
	&\int \vv^{2n}Q_\ns(G,F) \partial^\alpha_x\partial^\beta_v f \,dvdx
	\lesssim
	\|f\|_{H^{k,n}}
 \|Q_\ns(G,F)\|_{L^{2,n}}
 \lesssim
 \|f\|_{H^{k,n}}
 	\|G\|_{L^{\infty,3+\gamma+\epsilon}} \|F\|_{L^{2,n }}.
	\end{split}
	\]
	Using the Sobolev embedding theorem and \Cref{interpolation lemma}, we obtain the desired estimate
	\[
		\int \vv^{2n}Q_\ns(G,F) \partial^\alpha_x\partial^\beta_v f \,dvdx
			\lesssim \|g\|_{X^{k,n,m}} \|f\|_{X^{k,n,m}}^2.
	\]
	This concludes the proof of~\eqref{e.w05091} in this case.

	\smallskip

	\textbf{Case four:  $|\alpha''|+|\beta''|=k-3$ and $|\alpha'|+|\beta'|=3$.}
	The proof of~\eqref{e.w05091} in this case is exactly as in case three, except with the choices
	\[
		\theta \in \left(0, \min\left\{\frac{1}{2},\frac{5-4s}{6}\right\}\right),
			\quad
		p = \frac{3}{1+4\theta},
			\quad\text{and}\quad
		q = \frac{3}{2-4\theta}
	\]
	in the estimate of $Q_\s$.  As such, we omit the proof.

	\smallskip

	\textbf{Case five:  $|\alpha''|+|\beta''|= k-4$ and $|\alpha'|+|\beta'|= 4$.} 
	We begin with the singular term:
	\[
		\begin{split}
			\int \vv^{2n}Q_\s(G,F)\partial^{\alpha}_{x}\partial^{\beta}_{v}f\,dvdx
				&\lesssim \|\partial^{\alpha}_{x}\partial^{\beta}_{v}f\|_{L^{2,n}}
					\|Q_\s(G,F)\|_{L^{2,n}}
				\leq \|f\|_{H^{k,n}}
					\|Q_\s(G,F)\|_{L^{2,n}}.
		\end{split}
	\]
	It is clear that we need only bound the last term above.  Recalling \Cref{p:thm_2.4}.(iii), we have, for any $\mu \in ((2s-1)_+,1)$,
	\[
	\begin{split}
	\|Q_{s}(G,F)\|_{L_v^{2,n}} 
	&\lesssim
	\left( \|F\|_{L_{v}^{\infty, 3}} +\|\vv^{n+5/2+ \mu} F\|_{C_v^{1+\mu}}
	\right)
	\|G\|_{L_{v}^{2,n}}
	.
	\end{split}
	\]
	Applying the Sobolev embedding theorem with $\tilde m$ sufficiently large (depending only on $n$), we obtain, for $\epsilon = (1-\mu)/4$, 
	\[
		\|Q_{s}(G,F)\|_{L_v^{2,n}} 
			\lesssim
			\|F\|_{H_v^{5/2 + \mu + \epsilon, \tilde m}}
			\|G\|_{L_{v}^{2,n}}.
	\]
	Using H\"older's inequality and the Sobolev embedding theorem yields
	\[
	\begin{split}
		\|Q_\s(G,F)\|_{L^{2,n}}^2
			&\lesssim	\int
				\|F\|_{H_{v}^{5/2+\mu+\epsilon,\tilde m}}^{2}
				\|G\|_{L_{v}^{2,n}}^{2}\,dx
			\leq
				\|F\|^2_{L_{x}^{3}H_{v}^{5/2+\mu+\epsilon,\tilde m}}
				\|G\|^2_{L^{6}_{x}L_{v}^{2,n}}
			\\&\lesssim
				\|F\|_{H_{x}^{1/2} H_{v}^{5/2+\mu+\epsilon, \tilde m}}^2
				\|G\|_{H^{1,n}}^2
			\leq
				\|g\|_{H^{k,n}}^2
				\|f\|_{H^{k-1 + \mu+\epsilon, \tilde m}}^2.
	\end{split}
	\]
	To control the last term, we use the interpolation lemma, \Cref{interpolation lemma} and that, by construction, $\mu + \epsilon < 1$, to find
	\[
		\|f\|_{H^{k-1 + \mu, \tilde m}}
			\lesssim \|f\|_{X^{k,n,m}}
	\]
	as long as $m$ is sufficiently large depending only on $n$ and $s$.  This concludes the proof of the bound of the singular term.
	
	We now consider the non-singular part.  As above, it is enough to bound $\|Q_\ns(G,F)\|_{L^{2,n}}$.  To this end, applying \Cref{p:non-singular} yields
%
%
%
	\[
	\begin{split}
		\|Q_\ns(G,F)\|_{L^{2,n}}^2
		&\lesssim
			\left(
			\int\|F\|_{L_v^{\infty,n+3}}^{2}
			\|G\|_{L_v^{2,n}}^{2}\,dx\right)^{1/2}
		\lesssim
			\|F\|_{L^{\infty,n+3}}^2
			\|G\|_{L^{2,n}}^2
		\leq
			\|F\|_{L^{\infty,n+3}}^2
			\|g\|_{X^{k,n,m}}^2.
	\end{split}
	\]
	Thus, we need only bound the norm of $F$ on the right hand side.  By the Sobolev embedding theorem and the interpolation lemma \Cref{interpolation lemma}, we find
	\[
		\|F\|_{L^{\infty,n+3}}
			\lesssim \|F\|_{H^{7/2 + 1/4,n+3}}
			\lesssim \|f\|_{H^{k-1/4, n+3}}
			\lesssim \|f\|_{X^{k,n,m}},
	\]
	as long as $m$ is sufficiently large.  This concludes the proof of~\eqref{e.w05091} in case five.
		
	\smallskip
	
	\textbf{Case six:  $|\alpha''|+|\beta''|\leq k-5$ and $|\alpha'|+|\beta'|\geq 5$.} 
	We begin by bounding the term with $Q_\s$.  As above, it is enough to bound $Q_\s(G,F)$ in $L^{2,n}$.  First, by \Cref{p:thm_2.4}.(iii) with $\mu \in ((2s-1)_+,1)$, we find
	\[
	\|Q_\s(G,F)\|_{L^{2,n}}^2
	\lesssim
		\int \|G\|_{L_{v}^{2,n}}^{2}
		\left(
		\|F\|_{L_{v}^{\infty, m}}^2
		+\|\vv^{n + 5/2 + \mu} F\|_{C^{1+\mu}}^2
		\right)\,dx.
	\]
	Applying the Sobolev embedding theorem and letting $\epsilon = (1-\mu)/2$, we obtain
	\[
		\begin{split}
		\|Q_\s(G,F)\|_{L^{2,n}}^2
			&\lesssim \int \|G\|_{L_{v}^{2,n}}^{2}
				\|F\|_{H^{5/2 + \mu, \tilde m}_{v}}^2
				\,dx
			\leq \|G\|_{L^{2,n}}^2 \|F\|_{L^\infty_x H_v^{5/2 + \mu, \tilde m}}^2\\
			&\lesssim \|G\|_{L^{2,n}}^2 \|F\|_{H^{4 + \mu + \epsilon, \tilde m}}^2
			\lesssim \|g\|_{H^{k,n}}^2 \|f\|_{H^{k-1+\mu+\epsilon, \tilde m}}^2,
		\end{split}
	\]
	where $\tilde m$ is a constant depending only on $n$.  The proof concludes as in the previous case by using the fact that $k-1 + \mu + \epsilon < k$ and \Cref{interpolation lemma}.

The estimate of the non-singular part $Q_{\ns}$ is the same as in the previous case and is thus omitted.  This concludes the proof of~\eqref{e.w05091} in case six and, thus, all cases.
	\end{proof}

Having established the bounds above, we now construct a solution.
	\begin{proposition}[Construction of solution in the linear equation]
	Fix $T>0$, a function $g \in Y^{k,n,m}_T$, and the initial data $0 \leq f_{\rm in}\in X^{k,n,m}$.  Then there exists $f \in Y^{k,n,m}_T$ such that 
	\begin{align}
	f_{t} + v\cdot \nabla_{x}f
	=
	Q(g,f)
	\end{align}
	and $f(0,\cdot,\cdot) = f_{\rm in}$.  Moreover, $f\geq 0$.
\end{proposition}
\begin{proof} 
	The proof of \cite[Proposition 3.3]{HST_boltz_wp} can be adapted verbatim as it requires only the established bounds in \cite[Proposition 3.2]{HST_boltz_wp} (the analogue of our \Cref{p:existence}).  The proof is composed of three steps: (1) due to the Laplacian in $\mathcal{L}_{\sigma,\epsilon,\delta}$, apply the Schauder estimates to establish boundedness of a linear operator involving of $\mathcal{L}_{\sigma,\epsilon,\delta}$; (2) apply the method of continuity to construct the solution  of $\mathcal{L}_{\sigma,\epsilon,\delta}f=0$ using the bounds from the previous step, and (3) use the {\em a priori} estimates from \Cref{p:existence} to deregularize.  Due to its similarity to \cite[Proposition 3.3]{HST_boltz_wp}, we omit the details.
\end{proof}

\begin{proof}[Proof of existence in \Cref{t.main}]
	The idea used to prove \cite[Theorem 2.1]{HST_boltz_wp} is to construct a sequence $f_{i}$ solving
	\[
		(\partial_t + v\cdot \nabla_x) f_i
			= Q(f_{i-1},f_i),
	\]
	establishing the boundedness of this sequence inductively, and taking the limit $i\to\infty$.  Notice that we have the same bounds in \Cref{p:existence} as in \cite[Proposition 3.2]{HST_boltz_wp}, which is the crux of argument.  Thus, the proof in our setting will be unchanged and we omit the details.
\end{proof}

\subsection{Proof of uniqueness in \Cref{t.main}}\label{uniqueness}

We now finish the proof of \Cref{t.main} by establishing uniqueness.

\begin{proof}[Proof of uniqueness in \Cref{t.main}]
Consider any two solutions $f$ and $g$ of~\eqref{e:Boltzmann} with $f(0,\cdot,\cdot)= g(0,\cdot,\cdot) = f_{\rm in}$ and set $h=f-g$.  We have
\begin{align}\label{e.w05263}
h_{t} + v\cdot \nabla_{x} h
=
Q(f,h)+Q(h,g).
\end{align}
Then, we multiply~\eqref{e.w05263} by $\vv^{2n}h$ and integrate with respect to $v$ and $x$, yielding
\be\label{e.w05264}
\begin{split}
	\frac{1}{2}\frac{d}{dt}\|h\|_{L^{2,n}}^2
	&=
	\int \vv^{2n}Q(f,h)h\,dvdx
	+
	\int \vv^{2n}Q(h,g)h\,dvdx
	=
	I_{1}+I_{2}
	,
\end{split}
\ee
where
\be
\begin{split}
	I_{1}
	&=
	\int \vv^{2n}Q_\s(f,h)h\,dvdx
	+
	\int \vv^{2n}Q_\ns(f,h)h\,dvdx
	=
	I_{11}+I_{12}
\end{split}
\ee
and
\be
\begin{split}
	I_{2}
	&=
	\int \vv^{2n}Q_\s(h,g)h\,dvdx
	+
	\int \vv^{2n}Q_\ns(h,g)h\,dvdx
	=
	I_{21}+I_{22}
	.
\end{split}
\ee
For $I_{11}$, \Cref{p:thm_2.4}.(iv) yields, for $\tilde m$ sufficiently large,
\be\label{e.c6111}
I_{11}
\lesssim
\int
\|h\|_{L_{v}^{2,n}}^{2} 
\|f\|_{L_{v}^{\infty,\tilde m}}
\,dx
\lesssim
\|h\|_{L^{2}_{x}L_{v}^{2,n}}^{2} 
\|f\|_{L^{\infty}_{x}L_{v}^{\infty,\tilde m}}
\lesssim
\|h\|_{L^{2,n}}^{2} 
\|f\|_{X^{k,n,m}}
.
\ee
For $I_{12}$, we apply \Cref{p:non-singular} to obtain
\be\label{e.c6112}
	I_{12}
		\lesssim
			\|h\|_{L^{2,n}}^{2}
			\|f\|_{L^{\infty,3+\gamma+\epsilon}} 
		\lesssim
			\|h\|_{L^{2,n}}^{2} 
			\|f\|_{X^{k,n,m}}
.
\ee
For $I_{21}$, fix $\alpha \in (2s, 2)$, $\epsilon = (2 - \alpha)/2$, and $\tilde m$ be sufficiently large and apply \Cref{p:thm_2.4}.(iii) to find
\be\label{e.c6113}
	\begin{split}
	I_{21}
		&\lesssim \int \|h\|_{L_{v}^{2,n}}^2
			\left(
			\|g\|_{L_{v}^{\infty, m}}
			+\|\vv^{\tilde m} g\|_{C^{\alpha}_{v}}
			\right) dx
		\lesssim
			\|h\|_{L^{2,n}}^{2} 
			\|g\|_{L^{\infty}_{x}H_v^{3/2 + \alpha, \tilde m}}\\
		&\lesssim
			\|h\|_{L^{2,n}}^{2} 
			\|g\|_{H^{3 + \alpha + \epsilon, \tilde m}}
		\lesssim
			\|h\|_{L^{2,n}}^{2} 
			\|g\|_{X^{k,n,m}}.
	\end{split}
\ee
For $I_{22}$, apply \Cref{p:non-singular} to find
\be\label{e.c6114}
	\begin{split}
		I_{22} 
			\lesssim \|h\|_{L^{2,n}}^2 \|g\|_{L^{\infty,n+3}}
			\lesssim \|h\|_{L^{2,n}}^2 \|g\|_{X^{k,n,m}}.
	\end{split}
\ee

Combining the estimates of $I_{11},I_{12},I_{21}$, and $I_{22}$, that is,~\eqref{e.c6111}-\eqref{e.c6114}, and recalling that $\|f\|_{X^{k,n,m}}, \|g\|_{X^{k,n,m}} \lesssim 1$, we find
\[
	\frac{d}{dt} \|h(t)\|_{L^{2,n}}^2
		\lesssim \|h(t)\|_{L^{2,n}}^2.
\]
The Gr\"onwall inequality and the fact that $h(0,\cdot,\cdot) = 0$ implies that $h=0$.  We deduce that $f=g$, concluding the proof.
\end{proof}

\section{Proof of the estimates on the collision operator $Q$}\label{s:collision_proofs}

\subsection{Proof of the refined estimate on $K_g$ \Cref{l: w0524}}

\begin{proof}
	We first show that $|v+w|\approx |v|+|w|$ for any $w \in (v-v')^{\perp}$.  The ``$\lesssim$'' inequality is clear, so we show the other inequality:
	\begin{align*}
	|v+w|^2
	&=
	|v|^2 +2v\cdot w +|w|^2
	=
	|v|^2 +2v'\cdot w +|w|^2
	\geq
	|v|^2 - \frac{1}{1-\theta}|v'|^{2}-(1-\theta) |w|^{2} +|w|^2
	\\&\geq
	|v|^2 - (1-\theta)|v|^{2}-(1-\theta) |w|^{2} +|w|^2
	=
		\theta |v|^2 + \theta|w|^2
	.
	\end{align*}
	In the second equality, we used that $(v-v')\cdot w = 0$, in the first inequality, we used Young's inequality, and in the second inequality, we used the hypothesis that $(1-\theta)|v| \geq |v'|$.

	Recalling~\eqref{e.carleman2} and changing variables, we have
	\[
		|v-v'|^{3+2s} K_g(v,v')
			\approx
				\int_{v+(v'-v)^{\perp}} g(w)|v-w|^{\gamma+2s+1} \,dw
			=
				\int_{(v'-v)^{\perp}} g(v+w)|w|^{\gamma+2s+1} \,dw.
	\]
	Clearly, it is enough to simply bound the integral on the right hand side.  Using that $|v+w|\approx |v| + |w|$, as established above, we see that
	\begin{equation}\label{e.c611}
	\begin{split}
	&\Big|\int_{(v'-v)^{\perp}}
	g(v+w)|w|^{\gamma+2s+1} \,dw\Big|
	\lesssim
		\|g\|_{L^{\infty,m}}
		\int_{(v'-v)^{\perp}}
	\frac{|v|^{\gamma + 2s + 1}}{\langle v+w \rangle^{m}} \,dw
	\\&\quad\lesssim
	\|g\|_{L^{\infty, m}}
	\int_{(v'-v)^{\perp}}
	\frac{|w|^{\gamma+2s+1}}{\langle v \rangle^{m}+\langle w \rangle^{m}} \,dw
	\\&\quad=
	\|g\|_{L^{\infty, m}}
	\int_{(v'-v)^{\perp}\cap B_{\vv}}
	\frac{|w|^{\gamma+2s+1}}{\langle v \rangle^{m}+\langle w \rangle^{m}} \,dw
	+
	\|g\|_{L^{\infty, m}}
	\int_{(v'-v)^{\perp}\cap B^{c}_{\vv}}
	\frac{|w|^{\gamma+2s+1}}{\langle v \rangle^{m}+\langle w \rangle^{m}} \,dw\\
	&\quad= \|g\|_{L^{\infty,m}} \left( I_1 + I_2\right).
	\end{split}
	\end{equation}
	
	For $I_1$, that is, $w \in B_{\vv}$, we use the fact that
	\begin{equation}
		\frac{|w|^{\gamma+2s+1}}{\langle v \rangle^{m}+\langle w \rangle^{m}}
	\lesssim
		|w|^{\gamma + 2s + 1}
		\langle v \rangle^{-m}
	.												
	\end{equation}
	Thus, we see (recall we are integrating over a subset of a two-dimensional hyperplane)
	\begin{equation}\label{e.w05261}
	\begin{split}
	I_1
		\lesssim
			\vv^{-m} \int_{(v'-v)^{\perp}\cap B_{\vv}}
			|w|^{\gamma + 2s + 1} \,dw
		\lesssim
			\langle v \rangle^{\gamma +2s+3-m }
	.
	\end{split}
	\end{equation}

	For $I_2$, that is, $w \in B^{c}_{\vv}$, we have
	\begin{align}
	\frac{|w|^{\gamma+2s+1}}{\langle v \rangle^{m}+\langle w \rangle^{m}}
	\lesssim	
	\frac{|w|^{\gamma+2s+1}}{\langle w \rangle^{m}}
	\lesssim		
	\langle w \rangle^{\gamma +2s+1-m }
	.			
	\end{align}
		Therefore, we get (again, recall, we are integrating over a subset of a two-dimensional hyperplane)
		\begin{equation}\label{e.w05262}
		\begin{split}
		I_2
		\lesssim
			\int_{(v'-v)^{\perp}\cap B^{c}_{\vv}}
			\langle w \rangle^{\gamma +2s+1-m } \,dw
		\lesssim
			\langle v \rangle^{\gamma +2s+3-m }
		.
		\end{split}
		\end{equation}
	Combining~\eqref{e.c611}, \eqref{e.w05261}, and \eqref{e.w05262}, we obtain the desired inequality, concluding the proof.
\end{proof}

\subsection{Commutator estimate: proof of \Cref{p:commutator}}
Before beginning, we require a helper lemmas concerning the weighted Sobolev norms.  While this result is somewhat elementary, we do not know of a reference.  
\begin{lemma}\label{p03}
	For $\tilde s \in (0,1)$, $R>0$, $\ell \geq 0$, and $\mathcal{D} = \{(v,v') \in \mathbb{R}^6: |v-v'| \leq \vv / R\}$, we have, for any $f \in H^{\tilde s, \ell}(\R^3)$,
	\be\label{e.c353}
	\int_{\mathcal{D}}
	\vv^{\ell} \frac{|f(v)-f(v')|^{2}}{|v-v'|^{3+2\tilde s}}\,dv'dv
	\lesssim
	\|f\|_{H^{\tilde s,\ell}}^2
	.
	\ee
\end{lemma}
Before beginning we remark briefly about the content of \Cref{p03}.  Recall that $\|f\|_{\dot H^{\tilde{s}}}=\int \frac{|f(v)-f(v')|^{2}}{|v-v'|^{3+2\tilde{s}}}\,dv'dv$ and, hence,
\[
\|f\|_{H^{\tilde s, \ell}}^2
= \|\vv^\ell f\|_{H^{\tilde s}}^2
= \|\vv^\ell f\|_{L^2}^2 + \int \frac{ |\vv^\ell f(v) - \vvp^\ell f(v')|^2}{|v-v'|^{3 + 2\tilde s}} dv dv'.
\]
The difference between the quantity above and the left hand side of~\eqref{e.c353} is now clear.
\begin{proof}
	To begin, we use the triangle inequality and that $(a+b)^2 \leq 2a^2 + 2b^2$ to find
	\begin{align*}
		\int_{\mathcal{D}}&
		\vv^{\ell} \frac{|f(v)-f(v')|^{2}}{|v-v'|^{3+2\tilde s}}\,dv'dv
		\lesssim
		\int\limits_{\mathcal D} \frac{|\vv^{\ell} f(v)-\langle v'\rangle^{\ell} f(v')|^{2} +|\vv^{\ell} -\langle v'\rangle^{\ell}|^{2}|f(v')|^{2}}{|v-v'|^{3+2\tilde s}}\,dv'dv
		\\&=
		\int\limits_{\mathcal D}\frac{|\vv^{\ell} f(v)-\langle v'\rangle^{\ell} f(v')|^{2} }{|v-v'|^{3+2\tilde s}}\,dv'dv
		+
		\int\limits_{\mathcal D} \frac{|\vv^{\ell} -\langle v'\rangle^{\ell}|^{2}|f(v')|^{2}}{|v-v'|^{3+2\tilde s}}\,dv'dv.
	\end{align*}
	The first term above is clearly bounded above by $\|f\|_{\dot H^{\tilde s,\ell}}$ simply by enlarging the domain of integration.  Hence, we consider only the second term.

	For $(v,v')\in \mathcal D$, we find, via Taylor's theorem, that $|\vv^\ell - \vvp^\ell|^2 \lesssim \vv^{2\ell-2} |v-v'|^2$. Thus,
	\[\begin{split}
	\int\limits_{\mathcal D} &\frac{|\vv^{\ell} -\langle v'\rangle^{\ell}|^{2}|f(v')|^{2}}{|v-v'|^{3+2s}}\,dv'dv
	\lesssim \int\limits_{\mathcal D} \frac{\vv^{2(\ell-1)}|f(v')|^{2}}{|v-v'|^{1+2s}}\,dv'dv.
	\end{split}\]
	Next, clearly there exists $\tilde R>0$ depending only on $R$ such that $\mathcal{D} \subset \{(v,v')\in \mathbb{R}^6: |v-v'| \leq \vvp/\tilde R\}$.  Additionally, $(v,v') \in \mathcal{D}$ implies that $\vv \approx \vvp$.  These two facts yield
	\[
	\begin{split}
	\int\limits_{\mathcal D} &\frac{|\vv^{\ell} -\langle v'\rangle^{\ell}|^{2}|f(v')|^{2}}{|v-v'|^{3+2s}}\,dv'dv
	\lesssim \int \vvp^{2(\ell-1)} |f(v')|^2 \int_{B_{\vvp/\tilde R}(v')} \frac{1}{|v-v'|^{1+2s}}\,dv dv'\\
	&\lesssim \int \vvp^{2(\ell-1)} |f(v')|^2 \vvp^{2-2s} dv'
	\lesssim \int \vvp^{2\ell} |f(v')|^2 dv'
	= \|f\|_{L^{2,\ell}},
	\end{split}
	\]
	which concludes the proof.
\end{proof}

\begin{proof}[Proof of \Cref{p:commutator}]
	We prove this using the characterization of the $L^2$-norm via duality; that is, fix any $h \in L^2(\R^3)$ and we estimate
	\[
		\int h \left( \vv^\ell Q_\s(g,f) - Q_\s(g,\vv^\ell f)\right) \,dv.
	\]

	For any $v$, let $R_{v}=\langle v \rangle/10$ and denote the diagonal strip
	\begin{equation}\label{e.mcD}
		\mathcal D=\{(v,v'): |v-v'|<R_{v} \}.
	\end{equation}
	Recalling~\eqref{e.carleman1}, we rewrite the quantity of interest as
	\be\label{e.c341}
	\begin{split}
		\int h &\left(\langle v \rangle^{\ell}Q_\s(g,f)- Q_\s(g, \langle v \rangle^{\ell}f) \right)\,dv
		=
		\int K_{g}(v,v')h(v) \left(\langle v \rangle^{l}f(v')- \langle v' \rangle^{l}f(v')  \right)\,dv'dv
		\\&=
		\int \limits_{\mathcal D} K_{g}(v,v')h(v) \left(\langle v \rangle^{\ell}f(v')- \langle v' \rangle^{\ell}f(v')  \right)\,dv'dv\\
 		&\qquad +
		\int\limits_{\mathcal D^c} K_{g}(v,v')h(v) \left(\langle v \rangle^{\ell}f(v')- \langle v' \rangle^{\ell}f(v')  \right)\,dv'dv
		=
		I_{1}+I_{2}.
	\end{split}
	\ee
	We estimate each of $I_1$ and $I_2$ in turn. 
	
	\smallskip

	{\bf Step one: bounding $I_2$.}  We further decompose $I_2$ as
	\be\label{e.c342}
	I_{2}
	=
	\int\limits_{\mathcal D^c} K_{g}(v,v')h(v) \langle v \rangle^{\ell}f(v')\,dv'dv
	-
	\int\limits_{\mathcal D^c} K_{g}(v,v')h(v) \langle v' \rangle^{\ell}f(v')\,dv'dv
	=
	I_{21}-I_{22}.
	\ee
	
	We first consider $I_{22}$.  Applying \Cref{l:2.3}, we find
	\begin{equation}\label{e.c2241}
		\begin{split}
			|I_{22}|
			&\lesssim \|g\|_{L^{\infty, m}}
			\int\limits_{\mathcal D^c} \frac{|h(v)|}{|v-v'|^{3+2s}}
			\langle v \rangle^{\gamma+2s+1}
			\langle v' \rangle^{\ell}|f(v')|\,dv'dv
			.
		\end{split}
	\end{equation}
	Rewriting the limits of integration, using that $|v-v'| \gtrsim \vv$, and applying Cauchy-Schwarz in $v'$ yields
	\[\begin{split}
	\int\limits_{\mathcal D^c} &\frac{|h(v)|}{|v-v'|^{3+2s}}
	\langle v \rangle^{\gamma+2s+1}
	\langle v' \rangle^{\ell}|f(v')|\,dv'dv
	= \int \langle v \rangle^{\gamma+2s+1} |h(v)|
	\left( \int_{B_{R_v}^c(v)} \frac{\langle v' \rangle^{\ell}|f(v')|}{|v-v'|^{3+2s}}
	dv' \right) dv\\
	&\lesssim \int \langle v \rangle^{\gamma-2} h(v)
	\left(\int
		\vvp^{-3 - 2\epsilon} dv'\right)^{1/2}
	\left(\int \langle v' \rangle^	{\ell+3+2\epsilon}f^{2}(v') dv'\right)^{1/2}
	dv.
	\end{split}\]
	Noticing that the integral involving $f$ is a weighted $L^2$ norm of $f$, the middle integral is finite, and combining this with~\eqref{e.c2241}, we obtain
	\begin{equation}\label{e.w311}
		\begin{split}
			|I_{22}|
			&\lesssim 
			\|g\|_{L^{\infty, m}}
			\|f\|_{L^{2, \ell+1}}
			\int\langle v \rangle^{\gamma - \frac{3}{2} }h(v) dv
			= \|g\|_{L^{\infty, m}}
			\|f\|_{L^{2, \ell+3/2+\epsilon}}
			\|h\|_{L^2}
			.
		\end{split}
	\end{equation}
	We now consider $I_{21}$.  Here, we split the integral as follows:
	\begin{align*}
		I_{21}
		&=
		\int\limits_{\mathcal D^{c}\cap \{|v'|\geq |v|/10 \}} K_{g}(v,v')h(v) \langle v \rangle^{\ell}f(v')\,dv'dv
		+
		\int\limits_{\mathcal D^{c}\cap \{|v'|\leq |v|/10 \}} K_{g}(v,v')h(v) \langle v \rangle^{\ell}f(v')\,dv'dv\\
		&=
		I_{211}+I_{212}
		.
	\end{align*}
	The estimate of $I_{211}$ reduces to the estimate $I_{22}$:
	\begin{align*}
		|I_{211}|
		&\leq
		\int\limits_{\mathcal D^{c}\cap \{|v'|\geq |v|/10 \}} |K_{g}(v,v')h(v) \langle v \rangle^{\ell}f(v')|\,dv'dv
		\\&\lesssim
		\int\limits_{\mathcal D^{c}\cap \{|v'|\geq |v|/10 \}} |K_{g}(v,v')h(v) \langle v' \rangle^{\ell}f(v')|\,dv'dv
		\leq 
		\int\limits_{\mathcal D^{c}} |K_{g}(v,v')h(v) \langle v' \rangle^{\ell}f(v')|\,dv'dv.
	\end{align*}
	The last term above is exactly the term we estimate in~\eqref{e.c2241}; hence,
	\begin{equation}\label{e.w312}
		|I_{211}|
		\lesssim 
		\|g\|_{L^{\infty, m}}
		\|f\|_{L^{2, \ell+3/2+\epsilon}}
		\|h\|_{L^2}
		.
	\end{equation}
	Turning to $I_{212}$, we get
	\begin{equation}\label{e.c2251}
		\begin{split}
			|I_{212}| 
			&\lesssim		
			\|g\|_{L^{\infty, m}}
			\int\limits_{\mathcal D^{c}\cap \{|v'|\leq |v|/10 \}} \frac{h(v) \langle v \rangle^{\ell+\gamma +2s+3-m }f(v')}{|v-v'|^{3+2s}}\,dv'dv	.
		\end{split}
		\ee
		After applying Cauchy-Schwarz to the integral in $v'$, a direct computation using that $\ell>3/2$ yields
		\be\label{e.c2252}
		\begin{split}
			&\int\limits_{\mathcal D^{c}\cap \{|v'|\leq |v|/10 \}} \frac{h(v) \langle v \rangle^{\ell+\gamma +2s+3-m }f(v')}{|v-v'|^{3+2s}}\,dv'dv\\
			&\qquad\leq
			\int \vv^{\ell + \gamma +2s+3-m}h(v)
			\left(\int_{B_{R_v}(v)^c \cap B_{|v|/10}} \frac{ \vvp^{-2\ell} dv'}{|v-v'|^{6+4s}}\right)^{1/2}
			\left(\int \vvp^{2\ell}f^{2}(v')\,dv'\right)^{1/2}
			dv
			\\&\qquad\lesssim
			\|f\|_{L^{2,\ell}}
			\int \vv^{\ell + \gamma -m}h(v)
			dv.
		\end{split}
		\ee
		Using that $m > \ell +\gamma+ \frac{3}{2}$, we conclude from~\eqref{e.c2251} and~\eqref{e.c2252} that
		\be\label{e.w313}
		|I_{212}|
		\lesssim
		\|g\|_{L^{\infty, m}}
		\|f\|_{L^{2, \ell}}
		\|h\|_{L^{1, \ell + \gamma - m}}
		\lesssim
		\|g\|_{L^{\infty, m}}
		\|f\|_{L^{2, \ell}}
		\|h\|_{L^{2}}.
		\ee
		Combining \eqref{e.w311}, \eqref{e.w312}, and \eqref{e.w313} we deduce that
		\be\label{e.c311}
		|I_2|
		\lesssim \|f\|_{L^{2,\ell+1}} \|g\|_{L^{\infty,m}} \|h\|_{L^2}.
		\ee
		This concludes step one.

		\smallskip

		{\bf Step two: bounding $I_1$.} 
		For notational convenience, let $W_\ell(v) = \vv^\ell$.  For any function $\psi$ and any velocities $v$ and $v'$, let $\delta \psi= \psi(v) - \psi(v')$ (we suppress the dependence on $v$ and $v'$ as no confusion will arise).  Then, we rewrite $I_1$ as
		\begin{align*}
			I_{1}
			=
			\int \limits_{\mathcal D} K_{g}(v,v')h(v) \delta f\, \delta W_\ell\,dv'dv
			+
			\int \limits_{\mathcal D} K_{g}(v,v')h(v) f(v) \delta W_\ell\,dv'dv
			=
			I_{11}+I_{12}.
		\end{align*}
		For $I_{11}$, we see, by the definition of the kernel $K_{g}$ and get that
		\be\label{e.c2253}
		I_{11}^2
		\lesssim \|g\|_{L^{\infty,m}}^2
		\left(\int \limits_{\mathcal D}  \frac{\vv^{\gamma + 2s + 1}}{|v-v'|^{3 + 2s}} |h(v)| |\delta f| |\delta W_\ell|\,dv'dv\right)^{2}.
		\ee
		Next, 
		applying the Cauchy-Schwarz inequality yields
		\be\label{e.c2254}
		\begin{split}
			&\left(\int \limits_{\mathcal D}  \frac{\vv^{\gamma + 2s + 1}}{|v-v'|^{3+2s}} |h(v)| |\delta f| |\delta W_\ell|\,dv'dv\right)^{2}\\
			&\quad \leq
			\left(\int \limits_{\mathcal D}  \frac{h(v)^2}{\vv^{2\mu}} |v-v'|^{-3+2\mu}\,dv'dv\right)
			\left(\int \limits_{\mathcal D} \vv^{2(\mu+\gamma + 2s + 1)} \frac{(\delta f)^2 (\delta W_\ell)^2}{|v-v'|^{3 + 4s +2\mu}} \,dv'dv\right).
		\end{split}
		\ee
		We first consider the integral involving $h$.  Recalling the definition of $\mathcal{D}$~\eqref{e.mcD}, we find
		\be\label{e.c2255}
		\begin{split}
			\int \limits_{\mathcal D}  h(v)^{2}|v-v'|^{-3 + 2\mu}\langle v\rangle^{-2\mu}\,dv'dv
			&\leq \int h(v)^2 \vv^{-2\mu} \int_{B_{R_v}(v)} |v-v'|^{-3 + 2\mu} dv' dv\\
			&\lesssim \int h(v)^2 dv
			= \|h\|_{L^2}^2.
		\end{split}
		\ee
		Next, we consider the second integral in~\eqref{e.c2254}.  
		Recall that $|v-v'|<R_{v}$ by the definition of $\mathcal{D}$,~\eqref{e.mcD}.  Hence, by Taylor's theorem, we have
		\[
		|\delta W_\ell|^2
		\lesssim
		|v'-v|^{2}\langle v \rangle^{2\ell-2}.
		\]
		Using this and \Cref{p03}, we find
		\be\label{e.c2256}
		\begin{split}
			\int \limits_{\mathcal D}
			\frac{\vv^{2(\mu + \gamma + 2s + 1)}(\delta f)^{2} (\delta W_\ell)^{2}}
			{|v-v'|^{3 + 4s + 2\mu}} \,dv'dv
			&\lesssim 
			\int \limits_{\mathcal D}
			\frac{\vv^{2(\mu + \gamma + 2s + \ell)}(\delta f)^2}{|v-v'|^{3 + 2(2s - 1 + \mu)}}\,dv'dv
			\lesssim \|f\|_{H^{2s-1 + \mu, \mu + \ell}}^2.
		\end{split}
		\ee
		We conclude by combining~\eqref{e.c2253}-\eqref{e.c2256} to obtain
		\be\label{e.w314}
		|I_{11}|
		\lesssim \|g\|_{L^{\infty,m}} \|h\|_{L^2} \|f\|_{H^{2s  - 1 + \mu, \mu + \ell+\gamma+2s}}
		\ee

		We consider now $I_{12}$.  Using a second order Taylor expansion of $W_\ell(v) =\vv^\ell$, we see that
		\begin{align*}
			I_{12}
			&=
			\int \limits_{\mathcal D} K_{g}(v,v')h(v) f(v) \delta W_\ell\,dv'dv	
			\\&=
			\int h(v) f(v)\int \limits_{B_{R_{v}(v)}}  K_{g}(v,v')\left((D_{v} W_\ell)|_v (v-v')
			+\frac{1}{2} (v-v')\cdot (D_{v}^{2} W_\ell)|_{\xi_{v,v'}} (v-v')\right)\,dv'dv
			\\&=
			I_{121} + I_{122},
		\end{align*}
		where $\xi_{v,v'} = tv' + (1-t)v$ for some $t \in [0,1]$.
		For $I_{121}$, we use \Cref{p.w3311}.\eqref{l.is3.7} to obtain
		\be\label{e.c351}
		I_{121}=0.
		\ee
		For $I_{122}$, we use that $|(D^2_v W_\ell)|_{\xi_{v,v'}}| \lesssim \vv^{\ell-2}$, due to the fact that $v' \in B_{R_v}(v)$, in order to find
		\[
		|I_{122}|
		\les
		\int_{\R^3} |h(v) f(v)|\int_{B_{R_{v}(v)}}  |K_{g}(v,v')|
		\vv^{\ell-2} |v-v'|^2\,dv'dv.
		\]
		Thus, we have by appealing to \Cref{l:new kernel}
		\begin{equation}\label{e.w315}
			\begin{split}
				|I_{122}|
				&\les
				\|g\|_{L^{\infty,m}}
				\int_{\mathbb R^3} |h(v) f(v)|  
				\vv^{\ell-2 + \gamma + 2s} \,dv
				\les
				\|g\|_{L^{\infty,m}}
				\|f\|_{L^{2, \ell-2 + \gamma + 2s}}
				\|h\|_{L^{2}}.
			\end{split}
		\end{equation}
		Combining~\eqref{e.c351} and~\eqref{e.w315} and the fact that $\ell+3/2+\epsilon > \ell-2+\gamma + 2s$, we find
		\be\label{e.c352}
		|I_{12}|
		\lesssim \|g\|_{L^{\infty,m}}
		\|f\|_{L^{2, \ell+1}}
		\|h\|_{L^{2}}.
		\ee
		Thus, by~\eqref{e.w314} and \eqref{e.c352},
		\[
		|I_1|
		\lesssim \left(\|f\|_{L^{2,\ell+3/2+\epsilon}}+\|f\|_{H^{2s-1+\mu,\mu + \ell+\gamma+2s}}\right)
		\|g\|_{L^{\infty,m}}
		\|h\|_{L^2}.
		\]
		This concludes step two, and, thus, the proof.
	\end{proof}

\subsection{Collection of $Q_\s$ estimates: proof of \Cref{p:thm_2.4}.(i)-(iv)}

\subsubsection{Proof of \Cref{p:thm_2.4}.(i)}
\begin{proof}
%
%
	Let
	\be\label{e.c6151}
		\hat K_g(v,v')
			= \frac{1}{\|g\|_{L^{\infty, 3+\gamma+2s+\epsilon}}} K_g(v,v'),
	\ee
	and we have that $\hat K_g$ satisfies the conditions (4.2), (4.3), and (4.4) in \cite[Section 4]{IS2020weakharnack} uniformly in $v$. 
	This allows us to apply their general estimates, which we do now.  For clarity, we adopt their notation as closely as possible. 
	
	Let $\hat L_g$ be the operator defined by replacing the kernel $K_g$ with $\hat K_g$ in $Q_\s$, and let $\hat L_g^t$ be its transpose. 
	Letting $\Delta_i$ be the Littlewood-Paley projectors as in \cite[Proof of Theorem 4.1]{IS2020weakharnack} and using \cite[Theorems 4.3 and 4.6]{IS2020weakharnack}, yields, for any $\theta$,
	\[
		\|\hat L_g \Delta_i f\|_{L^2} \lesssim 2^{i \theta} \|\Delta_i f\|_{H^{2s-\theta}}
			\quad\text{and}\quad
		\|\hat L_g^t \Delta_i h\|_{L^2} \lesssim 2^{i (2s-\theta)} \|\Delta_i h\|_{H^\theta}.
	\]
	Also, recall that $\|\Delta_i \phi\|_{H^\theta} \approx 2^{i\theta} \|\Delta_i \phi\|_{L^2}$ for any $\theta$, $i$ and $\phi$.
	
	Using all estimates above for any fixed $\theta \in (0,2s)$ yields 
	\[
		\begin{split}
		&\frac{1}{\|g\|_{L^{\infty,3+\gamma+2s+\epsilon}}} \int Q_\s(g,f) h dv
			= \frac{1}{\|g\|_{L^{\infty,3+\gamma+2s+\epsilon}}} \sum_{i,j} \int Q_\s(g,\Delta_if)\Delta_jh dv\\
			&\qquad= \sum_{i\leq j} \int (\hat L_g \Delta_if) \Delta_jh dv
				+ \sum_{\theta i > (2s-\theta) j} \int \Delta_i f (\hat L_g^t \Delta_j h) dv\\
			&\qquad\lesssim \sum_{\theta i\leq (2s-\theta) j} 2^{\theta i - (2s-\theta)j} \|\Delta_i f\|_{H^{2s-\theta}} \|\Delta_j g\|_{H^\theta}
				+ \sum_{\theta i > (2s-\theta) j} 2^{-\theta i + (2s-\theta)j} \|\Delta_i f\|_{H^{2s-\theta}} \|\Delta_j g\|_{H^\theta}\\
			&\qquad= \sum_{i,j} 2^{-|\theta i - (2s-\theta)j|} \|\Delta_i f\|_{H^{2s-\theta}} \|\Delta_j g\|_{H^\theta}\\
			&\qquad\leq \Big(\sum_{i,j} 2^{-|\theta i - (2s-\theta)j|} \|\Delta_i f\|^2_{H^{2s-\theta}}\Big)^{1/2} \Big(\sum_{i,j} 2^{-|\theta i - (2s-\theta)j|} \|\Delta_j g\|^2_{H^\theta}\Big)^{1/2}
			\lesssim \|f\|_{H^{2s-\theta}} \|g\|_{H^\theta}.
		\end{split}
	\]
	In the last inequality, we sum first over $j$, using that $\theta, 2s-\theta>0$ by assumption, and then recalling that $\sum_i \|\Delta_i f\|^2_{H^{2s-\theta}} \approx \|f\|_{H^{2s-\theta}}^2$ (and similarly for $g$).
\end{proof}

\subsubsection{Proof of \Cref{p:thm_2.4}.(ii)}

\begin{proof}
We adopt the notation and setting of the proof of \Cref{p:thm_2.4}.(i).  Then
\[
	\begin{split}
		\frac{1}{\|g\|^{1/2}_{L^{\infty,3+\gamma+2s+\epsilon}}} \|Q_\s(g,f)\|_{L^2}
			&= \|\hat L_g f\|_{L^2}
			\leq \sum_{i=0}^\infty \|\hat L_g \Delta_i f\|_{L^2}
			\lesssim \sum_{i=0}^\infty 2^{-i\theta} \|\Delta_i f\|_{H^{2s+\theta}}\\
			&\lesssim \Big(\sum_{i=0}^\infty \|\Delta_i f\|^2_{H^{2s+\theta}}\Big)^{1/2}
			\approx \|f\|_{H^{2s+\theta}}.
	\end{split}
\]
%
%
\end{proof}

\subsubsection{Proof of \Cref{p:thm_2.4}.(iii)}

In order to establish part (iii) of \Cref{p:thm_2.4}, we require an analogue of Young's convolution inequality in the setting of the weighted Lebesgue spaces in order to handle terms of the form $\int g(w) |v-w|^{\gamma + 2s} dw$.  These have been well-studied and are understood in some generality (see, e.g., \cite{GuoFanWuZhao}).  However, for the convenience of the reader and because we can get a slightly sharper estimate (due to the specific form considered here), we include the proof.

\begin{lemma}[Weighted Young's inequality]
\label{l: Weighted Young}
	Suppose that $n > 3/2 + \eta$, $-3<\eta<0$, and $\ell > 3/2 + \eta + (3/2 - n)_+$.  If $g \in L^{2,n}$, then
	\be
	\int \vv^{-2\ell}\left(\int g(\tilde v)|v-\tilde v|^{\eta}\,d\tilde v \right)^{2}\,dv
	\lesssim
	\|g\|_{L^{2,n}}^2
	.
	\ee
\end{lemma}
\begin{proof}
	For succinctness, we let $A(v) = |v|^{\eta}$ and, without loss of generality we assume that $g \geq 0$.  First, we decompose the integral on the left hand side yielding
	\[
	\label{e.w3261}
	\begin{split}
	\int
	\vv^{-2\ell} (g*A)^2\, dv
	&\leq \int
	\vv^{-2\ell}
	\Big(\int_{B_{R_v}(v)} g(v') A(v-v')\,dv'
	\Big)^{2}\,dv
	\\&\quad+
	\int
	\vv^{-2\ell}
	\Big(\int_{B_{R_v}^c(v)} g(v') A(v-v')\,dv'
	\Big)^{2}\,dv
	= I_{1}+I_{2}
	.
	\end{split}
	\]
	For $I_{1}$, we use Cauchy-Schwarz inequality to obtain
	\[
	\begin{split}\label{c.w3262}
	I_{1}
	&\leq
	\int
	\vv^{-2\ell}
	\Big(\int_{B_{|v|/10}(v)} g(v')^{2}A(v-v')\,dv'
	\Big)
	\Big(\int_{B_{|v|/10}(v)} A(v-v')\,dv'
	\Big)\,dv
	\\&\lesssim
	\int
	\vv^{-2\ell+3+\eta}
	\Big(\int_{B_{|v|/10}(v)} g(v')^{2}A(v-v')\,dv'
	\Big)\,dv
	.
	\end{split}
	\]
	For $v'\in B_{|v|/10}(v)$, we have $\vvp \approx \vv$ and 
	$v\in B_{|v'|/2}(v')$.  Therefore,
	\[
	\begin{split}\label{c.w3263}
	&I_{1}
	\lesssim
	\int
	|g(v')|^{2}
	\int_{B_{|v'|/2}(v')} \vv^{-2\ell+3+\eta}A(v-v')\,dv
	\,dv'
	\\&\lesssim
	\int
	|g(v')|^{2}\vvp^{-2\ell+3+\eta}
	\int_{B_{|v'|/2}(v')} A(v-v')\,dv
	\,dv'
	\lesssim
	\int
	g(v')^{2}\vvp^{-2\ell+2(3+\eta)}\,dv'
	\lesssim
	\|g\|_{L^{2,n}}^2
	,
	\end{split}
	\]
	where we used that $-\ell+(3+\eta)\leq n$.
	For $I_{2}$, we apply the Cauchy-Schwarz inequality to find
	\[
	\begin{split}
	I_{2}
	&\leq
	\int
	\vv^{-2\ell}
	\Big(\int_{B^{c}_{|v|/10}(v)} \vvp^{2n}|g(v')|^{2}\,dv'
	\Big)
	\Big(\int_{B^{c}_{|v|/10}(v)} \vvp^{-2n}|v-v'|^{2\eta}\,dv'
	\Big)\,dv
	\\&\lesssim
	\|g\|_{L^{2,n}}^2
	\int
	\vv^{-2\ell}
	\Big(\int_{B^{c}_{|v|/10}(v)} \vvp^{-2n}|v-v'|^{2\eta}\,dv'
	\Big)\,dv
	\lesssim
	\|g\|_{L^{2,n}}^2
	\int
	\vv^{-2\ell+(3-2n)_+ + 2\eta}\,dv
	.
	\end{split}
	\]
	We conclude by using the conditions on $n$ and $\ell$. These were also used in the last inequality. Combining the estimates of $I_1$ and $I_2$ finishes the proof.
\end{proof}

We are now able to prove \Cref{p:thm_2.4}.(iii).  
\begin{proof}[Proof of \Cref{p:thm_2.4}.(iii)]
The proof is somewhat close to that of \cite[Proposition~3.1.(i)]{HST_boltz_wp}, so we omit details where steps are similar.  We may, without loss of generality, assume that $\alpha \in (0,1) \cup (1,2)$.  If not, we may simply take $\alpha' < \alpha$ such that $\alpha' \in (0,1) \cup (1,2)$ and use that $C^{\alpha} \xhookrightarrow{} C^{\alpha'}$.  Finally, the proof is simpler when $\alpha < 1$; hence, we consider only the case $\alpha \in (1,2)$.

We begin with an annular decomposition: let $A_{k}(v)=B_{2^{k}|v|} (v)\backslash B_{2^{k-1}|v|}(v)$ and write:
\be
	Q_\s(g,f)
		=
		\sum_{k\in \mathbb Z}
		\int_{A_{k}(v)} K_{g}(v,v') (f(v')-f(v))\,dv'
		.
\ee
Let $\bar{\mu}=n+5/2 +(3/2-n)_+ +\alpha + \gamma+\epsilon$.

\smallskip
\noindent
{\bf Step One: estimating the sum for any $k\leq 1$.}  By using a Taylor expansion, we see
\[
	f(v')-f(v)
		=
		((Df)(\xi_{v',v})-(Df)(v))\cdot (v'-v)+(Df)(v)\cdot (v'-v)
		,
\]
where $\xi_{v,v'} = tv' + (1-t)v$ for some $t \in [0,1]$. Thus, by \Cref{p.w3311}.(\ref{l.is3.4}) and (\ref{l.is3.7}),
\[
	\begin{split}
		&\Big|\int\limits_{A_{k}(v)} K_{g}(v,v') (f(v')-f(v))\,dv'\Big|
		\lesssim
			\vv^{-\bar{\mu}}(2^{k}|v|)^{\alpha-2s}
			\|\langle \cdot \rangle^{\bar{\mu}}Df\|_{C^{\alpha}}
			\int |g(v')||v-v'|^{\gamma+2s}\,dv'
			.
	\end{split}
\]
Recalling that $\alpha - 2s > 0$, by assumption, we have that $|v|^{\alpha - 2s} \leq \vv^{\alpha-2s}$.  Hence,
\begin{align*}
	&\int\Big(\int_{A_{k}(v)}\vv^{n} K_{g}(v,v') (f(v')-f(v))\,dv'\Big)^2\,dv
		\\&\quad\lesssim
		2^{2k(1+\alpha-2s)} \|\langle \cdot \rangle^{\bar{\mu}}Df\|_{C^{\alpha}}^{2}
		\int	
		\vv^{-2(\bar{\mu} - n - 1 - \alpha + 2s)}
		\left(\int |g(v')||v-v'|^{\gamma+2s}\,dv'
		\right)^{2}\,dv
		.
\end{align*} 
We are now in a position to apply the weighted Young's convolution inequality \Cref{l: Weighted Young}.  Indeed, by construction, $\ell:=\bar\mu - n - 1 - \alpha + 2s$ and $n$ satisfy the conditions of \Cref{l: Weighted Young} so that
\[
	\int\Big(\int_{A_{k}(v)}\vv^{n} K_{g}(v,v') (f(v')-f(v))\,dv'\Big)^2\,dv
		\lesssim 2^{2k(1+\alpha-2s)} \|\langle \cdot \rangle^{\bar{\mu}}Df\|_{C^{\alpha}}^{2} \|g\|_{L^{2,n}}^2.
\]

\smallskip
\noindent
{\bf Step Two:  estimating the sum for $k \geq 0$ when $|v'| \geq \vv/2$}. By \Cref{p.w3311}.(\ref{l.is3.4}),
\[
	\begin{split}
		\Big|\int\limits_{A_{k}\backslash B_{\vv/2}} &K_{g}(v,v') (f(v')-f(v))\,dv'\Big|
		\lesssim
		\vv^{-m}\|f\|_{L^{\infty,m}}
		\int\limits_{A_{k}\backslash B_{\vv/2}} |K_{g}(v,v')| \,dv'
		\\&\lesssim
		\vv^{m}\|f\|_{L^{\infty,m}}(2^{k}\vv)^{-2s}
		\left(\int|g(v')||v-v'|^{\gamma+2s}\,dv' \right)		
		.
	\end{split}
\]
Then, similar to Step One, we apply \Cref{l: Weighted Young} to obtain
\begin{equation*}
	\begin{split}
		&\int \vv^{2n}\left|\sum_{k\geq 0}\int_{A_{k}\backslash B_{\vv/2}} K_{g}(v,v') (f(v')-f(v))\,dv'\right|^2\,dv
		\\&\lesssim
		\|f\|_{L^{\infty,m}}^{2}
		\Big(\sum_{k\geq 0}2^{-2ks}\Big)^{2}
		\int \vv^{2(n-m-2s)}\left(\int |g(v')||v-v'|^{\gamma+2s}\,dv'\right)^{2}\,dv
		\lesssim
		\|f\|_{L^{\infty,m}}^{2}
		\|g\|_{L^{2,n}}^{2}		
		,
	\end{split}
\end{equation*}
where we used $n>3/2+\gamma+2s$ and $m>3/2+\gamma+(3/2-n)_{+}$.

\smallskip

\noindent{\bf Step Three: estimating the sum for $k \geq 0$ when $|v|\leq 10$ and $|v'| \leq \vv/2$.} This is similar to Step One.  The benefit is we are integrating over a compact set in $v$.  As such, we omit the proof and simply state that
\begin{align*}
	\int_{B_{10}}\Bigg(&\sum_{k\geq 0 }\int_{A_{k}\cap B_{\vv/2}}\vv^{n} K_{g}(v,v') (f(v')-f(v))\,dv'\Bigg)^2\,dv
		\\&=
		\int_{B_{10}}\Bigg(\int_{B_{\vv/2}\setminus B_{|v|/2}}\vv^{n} K_{g}(v,v') (f(v')-f(v))\,dv'\Bigg)^2\,dv
		\lesssim
		\|Df\|_{C^{\alpha}}^{2}
		\|g\|_{L^{2,n}}^{2}
		.
\end{align*}
Hence, we proved Step Three.
		
		\smallskip
		
\noindent{\bf Step Four: estimating the sum for $k \geq 0$ when $|v|\geq 10$ and $|v'| \leq \vv/2$.}    For any $|v|\geq 10$,
\[
	\begin{split}
		\Big|\sum_{k\geq 0}&\int_{A_{k}\cap B_{\vv/2}} K_{g}(v,v') (f(v')-f(v))\,dv'\Big|\\
		&\lesssim
		\int_{ B_{\vv/2}} K_{|g|}(v,v') |f(v')|\,dv'
		+
		\int_{ B_{\vv/2}} K_{|g|}(v,v') |f(v)|\,dv'
		=
		I_{1} + I_{2}
		.
	\end{split}
\]
For $I_{2}$, we notice that $B_{\vv/2} \subseteq (B_{2\vv}(v) \backslash B_{\vv/4}(v))$ due to the fact that $|v|\geq 10$. Then by \Cref{p.w3311}.\eqref{l.is3.4}, we have
\[
	\begin{split}
		\int_{ B_{\vv/2}} K_{|g|}(v,v') \,dv'
		&\lesssim
		\int_{ B_{2\vv}(v) \backslash B_{\vv/4}(v)} K_{|g|}(v,v')\,dv'
		\lesssim
		\vv^{-2s} \int|g(v')||v-v'|^{\gamma+2s}\,dv'
		.
	\end{split}
\]
Applying the weighted Young's inequality \Cref{l: Weighted Young} yields
\[
	\begin{split}
		\int \vv^{2n} I_{2}^{2}\,dv
		&\lesssim
		\int \vv^{2(n-2s)}|f(v)|^{2} \left(\int |g(v')||v-v'|^{\gamma+2s}\,dv' \right)^{2}\,dv
		\\&\lesssim
		\|f\|_{L^{\infty,m}}^{2}
		\int \vv^{-2m+2n-4s} \left(\int |g(v')||v-v'|^{\gamma+2s}\,dv' \right)^{2}\,dv
		\lesssim
		\|f\|_{L^{\infty,m}}^{2}
		\|g\|_{L^{2,n}}^{2}
	,
	\end{split}
\]
as desired. 
For $I_{1}$, the proof is omitted as it is exactly as in \cite[Proposition~3.1.(i)]{HST_boltz_wp}.  This finishes the proof.
\end{proof}

\subsubsection{Proof of \Cref{p:thm_2.4}.(iv)}

	\begin{proof}  Without loss of generality, we assume that $f,g\geq 0$.   Let $F=\vv^{n}f(v)$. We see
	\[
	\begin{split}
	\int \vv^{2n} Q_{s}(g,f)f\,dv
	=
	\int F Q_{s}(g,F)\,dv
	+
	\int F[\vv^{n} Q_{s}(g,f)-Q_{s}(g,F)]\,dv
	=
	I_{1}+I_{2}
	.
	\end{split}
	\]
	We further decompose $I_1$ into three parts:
	\[
	\begin{split}
	I_{1}
	&=
	-\int [F(v)-F(v')]^{2}K_{g}(v,v')\,dv'dv
		+
		\int \int [K_{g}(v,v')-K_{g}(v',v)]F(v)F(v')\,dv'\,dv\\
	&\qquad
		- \int [K_{g}(v,v')-K_{g}(v',v)]F(v')^2\,dv'\,dv
		= I_{11} + I_{12} + I_{13}.
	\end{split}
	\]
	The first term, $I_{11}$, has a good sign (and is used for cancellation below). 
The integrand in $I_{12}$ is antisymmetric with respect to the ``pre-post change of variables'' $(v,v') \mapsto (v',v)$, so $I_{12}=0$.  To estimate $I_{13}$, we use \Cref{p.w3311}.\eqref{l.is3.6}.  Hence, we find
\[
	|I_{12}|
	\lesssim
		\int F(v') \int g(z) |z-v'|^{\gamma + 2s} dz dv'
	\lesssim
		\|g\|_{L^{\infty,m}} \|F\|_{L^2}
		= \|g\|_{L^{\infty,m}} \|f\|_{L^{2,n}}.
	\]
	Here we used that $m > 3 + \gamma + 2s$ and $\gamma + 2s \leq 0$.  This concludes the bound on $I_1$.

	For $I_{2}$, we apply Young's inequality to find
	\[
	\begin{split}
	&I_{2}
	=
	\int F(v)f(v')K_{g}(v,v')(\vv^{n}-\vvp^{n})\,dv' dv
	\\&=
	\int (F(v)-F(v'))f(v')K_{g}(v,v')(\vv^{n}-\vvp^{n})\,dv' dv
	+
	\int F(v')f(v')K_{g}(v,v')(\vv^{n}-\vvp^{n})\,dv' dv
	\\&\leq
	-\frac{1}{2}I_{11}
	+
	\frac{1}{2}\int f^{2}(v')K_{g}(v,v')(\vv^{n}-\vvp^{n})^{2}\,dv' dv
	+
	\int F(v')f(v')K_{g}(v,v')(\vv^{n}-\vvp^{n})\,dv' dv.
	\end{split}
	\]
	Define the last two integrals to be $I_{21}$ and $I_{22}$.  The argument for $I_{21}$ is similar to and easier than $I_{22}$; hence, we omit it.
	
	We now bound $I_{22}$.  To do so, we split the integral into domains of integration $\mathcal D$, $\mathcal D^c \cap \{|v| \leq 10 |v'|\}$, and $\mathcal D^c \cap \{|v| \geq 10 |v'|\}$, where $\mathcal D = \{(v,v'): 10|v-v'| \leq \min\{\vv, \vvp\}\}$.  We denote the resulting integrals $I_{221},$ $I_{222}$, and $I_{223}$, respectively.  
	
	Considering $I_{221}$ first, we use a Taylor expansion, \Cref{p.w3311}.\eqref{l.is3.72}, \Cref{l:new kernel}, and the fact that $\vv \approx \vvp$ to find, for $\xi$ between $v$ and $v'$
	\[
	\begin{split}
	|I_{221}|
	&\leq \int \frac{F(v')^2}{\vvp} \int\limits_{B_{\vvp/2}(v')}K_g(v,v') \left[(v-v') \cdot v' n\vvp^{n-2} + \frac{n \langle \xi\rangle^{n-2}}{2}(v-v') \cdot \left(\Id + \frac{\xi\otimes\xi}{|\xi|^2}\right)(v-v')\right] dv dv'\\
	&\lesssim \int \frac{F(v')^2}{\vvp} \Big|\int\limits_{B_{\vvp/2}(v')}K_g(v,v') (v-v') dv \Big| dv'
	+ \int \frac{F(v')^2}{\vvp^2} \int\limits_{B_{\vvp/2}(v')}K_g(v,v') |v-v'|^2 dv dv'\\
	&\lesssim \int \frac{F(v')^2}{\vvp} \int g(w)|v'-w|^{1+\gamma} dw dv'
	+ \int \frac{F(v')^2}{\vvp^{2s}} \int g(w) |v'-w|^{\gamma+2s} dw dv' 
	\lesssim \|g\|_{L^{\infty, m}} \|f\|_{L^{2,n}}^2.
	\end{split}
	\]
	Above, we used that $m > 3 + \gamma + 2s$.
	
	Next we consider $I_{222}$.  In this case $\vv \lesssim \vvp$; hence, using that \Cref{p.w3311}.\eqref{l.is3.4} and that $m > 3 + \gamma + 2s$ yields
	\[
	|I_{222}|
	\lesssim \int F(v')^2 \int\limits_{B_{\vvp/2}^c(v')} K_g(v,v') dv dv'
	\lesssim \int \frac{F(v')^2}{\vvp^{2s}} \int g(w) |v'-w|^{\gamma + 2s} dw dv'
	\lesssim \|g\|_{L^{\infty, m}}\|f\|_{L^{2,n}}^2.
	\]
	Finally, we handle $I_{223}$. Indeed, we use that $\vvp \lesssim \vv$, the definition of $\mathcal D$, and \Cref{l: w0524} to get
	\[\label{e.w0409}
	\begin{split}
	\Big|\int\limits_{\mathcal D^c \cap \{|v| \geq 10 |v'|\}} &F(v')f(v')K_{g}(v,v')(\vv^{n}-\vvp^{n})\,dv' dv\Big|
		\\&\lesssim 
			\int\limits_{\mathcal D^c \cap \{|v| \geq 10 |v'|\}} F(v')f(v') K_g(v,v')  \vv^n\, dv' dv
	\\&\lesssim 
		\int\limits_{\mathcal D^c \cap \{|v| \geq 10 |v'|\}} F(v')f(v') \frac{|v-v'|^{3+2s}}{\vvp^{3+2s}} K_g(v,v')  \vv^n\, dv' dv
	\\&\lesssim
	\|g\|_{L^{\infty,m}} \int F(v') f(v') \int\limits_{\{10 |v'|\leq |v|\}} \vv^{- m + 3 + \gamma + 2s + n} dv dv'
	\lesssim 
	\|g\|_{L^{\infty, m}} \|f\|_{L^{2,n}}^2.
	\end{split}
	\]
	In the last inequality, we used that $m > n + 6 + \gamma + 2s$.  This concludes the proof. 
\end{proof}

\subsection{Proof of \Cref{p:symmetry}}

In order to prove \Cref{p:symmetry}, we first state a useful estimate that follows from work in \cite{IS2020weakharnack}.
\begin{lemma}\label{l:w0513}
	For any measurable $g$, if $\gamma + 2s \leq 0$ and $\epsilon>0$, then
	\begin{equation}
		\begin{split}
			\Big|\int K_{g}(f'-f)^{2}\,dv'dv\Big|
			\lesssim
			\|g\|_{L^{\infty, 3+\gamma + 2s + \epsilon}}
			\|f\|_{H^s}^{2}
			.
		\end{split}
	\end{equation}
\end{lemma}
\begin{proof}
	Recall that $\hat K_g$, defined in~\eqref{e.c6151}, 
	satisfies the conditions (4.2), (4.3), and (4.4) in \cite[Section 4]{IS2020weakharnack} uniformly in $v$. Thus, applying \cite[Lemma 4.2]{IS2020weakharnack}, we find
	\[
	\begin{split}
	\Big|\int K_{g}(f'-f)^{2}\,dv'dv\Big|
	=
	\|g\|_{L^{\infty, 3+\gamma+2s+\epsilon}}
	\Big|\int \hat K_g(v,v') (f'-f)^{2}\,dv'dv\Big|
	\lesssim \|g\|_{L^{\infty,3 + \gamma + 2s + \epsilon}} \|f\|_{H^s}^2,
	\end{split}
	\]
	which concludes the proof.
\end{proof}

%
%
%
%
%
Now we prove \Cref{p:symmetry}. 
\begin{proof}[Proof of \Cref{p:symmetry}]
	We consider only the case $\partial=\partial_{v_i}$ for $i \in \{1,2,3\}$. The case when $\partial=\partial_{x_i}$ is similar and simpler as it commutes with $\vv^{2n}$. 
	First, let $F=\vv^{n}f$. Then
	\begin{equation*}
		\begin{split}
			\int \vv^{2n}Q_\s(g,f)\partial f\,dvdx
			&=
			\int [\vv^{n}Q_\s(g,f)-Q_\s(g,\vv^{n}f)]\vv^{n}\partial f\,dvdx
			-
			\int Q_\s(g,F)fnv_{i}\vv^{n-2}\,dvdx
			\\&\quad+
			\int Q_\s(g,F)\partial F\,dvdx
			=
			I_{1}+I_{2}+I_{3}
			.
		\end{split}
	\end{equation*}
	For $I_{1}$, we apply the commutator estimate \Cref{p:commutator} to get
	\begin{equation}
		\begin{split}\label{e.w04111}
			|I_{1}|
			&\lesssim
			\int
			(\|f\|_{L_v^{2,n+3/2 + \epsilon}} + \|f\|_{H_v^{2s -1 + \mu,\mu+ n + \gamma + 2s}})
			\|g\|_{L_v^{\infty, m}}
			\|\partial f\|_{L_v^{2,n}}
			\,dx
			\\&\lesssim
			\int
			(\|f\|_{L_v^{2,n+3/2 + \epsilon}} + \|f\|_{H_v^{2s -1 + \mu,\mu+ n + \gamma + 2s}})
			\|g\|_{L_v^{\infty, m}}
			\|f\|_{H_v^{1,n}}
			\,dx
			\\&\lesssim
			\|g\|_{L^{\infty, m}}
			\left(
			\|f\|_{L^{2,n + 3/2 + \epsilon}}
			+
			\|f\|_{H^{2s -1 + \mu,\mu+n + \gamma + 2s}}
			\right)
			\|f\|_{H^{1,n}}.
		\end{split}
	\end{equation}
	
	To estimate $I_{2}$, we apply \Cref{p:thm_2.4}.(i) with $\theta=1$ if $s > 1/2$ or $\theta = 2s-1+\mu$ if $s \leq 1/2$ to find
	\be\label{e.w04112}
	\begin{split}
	|I_{2}|
	&\lesssim
	\int
	\|g\|_{L_v^{\infty, m}}  \|F\|_{H_v^{\theta}}\|fnv_{i}\vv^{n-2}\|_{H_v^{2s-\theta}}
	\,dx
	\\&\lesssim
	\int
	\|g\|_{L_v^{\infty, m}}  \|f\|_{H_v^{\theta,n}}\|f\|_{H_v^{2s-\theta, n-1}}
	\,dx
	\lesssim \|g\|_{L^{\infty, m}}  \|f\|_{H^{\theta,n}}\|f\|_{H^{2s-\theta, n-1}}.
	\end{split}
	\ee
	Using the choice of $\theta$, the right hand side above is less than or equal to (up to a constant) the right hand side of~\eqref{e.w04111}.  
	
	We decompose $I_3$ into two parts:
	\[
		\begin{split}
			I_{3}
			&=
			\int Q_\s(g,F)\partial F\,dv dx\\
			&=
			\int K_{g}(F'-F)(\partial F - (\partial F)')\,dv'dv dx
			+
			\int (K_{g}-K'_{g})(F'-F)\partial F\,dv'dv dx
			=
			I_{31}+I_{32}
			.
		\end{split}
	\]
	For $I_{31}$, we manipulate by integration-by-parts and apply \Cref{l:w0513} to find
\[
		\begin{split}
			|I_{31}|
			&=
			\Big|\int K_{g}(\partial+\partial')(F'-F)^{2}\,dv'dv dx\Big|
			=
			\Big|\int (\partial+\partial')K_{g}(F'-F)^{2}\,dv'dv dx\Big|
			\\&=
			\Big|\int K_{\partial g}(F'-F)^{2}\,dv'dvdx\Big|
			\lesssim
			\int \|\partial g\|_{L_v^{\infty, 3+\gamma + 2s + \epsilon}}
			\|F\|_{H_v^s}^{2} dx
			=
			\int \|\partial g\|_{L_v^{\infty, 3+\gamma + 2s + \epsilon}}
			\|f\|_{H_v^{s,n}}^{2} dx
			.
		\end{split}
\]
	Fix the conjugate exponents $p=3/2(1-s)$ and $q=3/(2s+1)$.  Applying H\"older's inequality and the Sobolev embedding theorem yields
	\[
	\begin{split}
	|I_{31}|
	&\lesssim 
	\|\partial g\|_{L^p_x L_v^{\infty,3+\gamma + 2s + \epsilon}} \|f\|_{L^{2q}_x H_v^{s,n}}^2\\
	&\lesssim 
	\|\partial g\|_{H^{(2s - 1/2)_+}_x H_v^{3/2+\epsilon,3+\gamma + 2s + \epsilon}} \|f\|_{H^{1-s}_x H_v^{s,n}}^2
	\lesssim 
	\|\partial g\|_{H^{3/2 + \epsilon + (2s-1/2)_+,3+\gamma + 2s + \epsilon}} \|f\|_{H^{1,n}}^2.
	\end{split}
	\]

The term $I_{32}$ is considered in \cite[Proposition 3.1.(iv), estimate of $I_2$]{HST_boltz_wp}.  A close inspection of the proof shows that it applies in our setting.  Hence, for simplicity, we cite directly that, for any $\mu \in (s,\min\{2s,1\})$,
	\begin{align*}\label{e.w04092}
		|I_{32}|
		\lesssim
		\| g\|_{C^{\mu, 3+\epsilon}}
		\|f\|_{H^{s,n + 3/2 + \epsilon + (\gamma + 2s + 1)_+}}
		\|f\|_{H^{1, n}}
		.
	\end{align*}
	Combining the above estimates of $I_{31}$ and $I_{32}$ together yields
	\begin{equation}\label{e.w04093}
		\begin{split}
			|I_{3}|
			\lesssim
			\|\partial g\|_{H^{3-s,3}} \|f\|_{H^{1,n}}^2
			+
			\| g\|_{C^{\mu, 3+\epsilon}}
			\|f\|_{H^{s,n + 3/2 + \epsilon + (\gamma + 2s + 1)_+}}
			\|f\|_{H^{1, n}}
			.
		\end{split}
	\end{equation}
	
	The proof is finished after combining~\eqref{e.w04111}, ~\eqref{e.w04112}, and~\eqref{e.w04093}.
\end{proof}



\section{A simple proof of local well-posedness when $0<s<1/2$: \Cref{t:C1_wellposed}}\label{s:simple}

Here we provide a short proof of local well-posedness when $s\in(0,1/2)$, taken as a standing assumption throughout the section even when not explicitly stated.  As many of the technical details are exactly the same as in the proof of \Cref{t.main}, we only outline the main points.  As the proof is the same for $k>1$, we show only the $k=1$ case.  Thus, we simplify the notation using $\tilde X^{m_0,m_1}$ in place of $\tilde X^{1,m_0,m_0}$ (the definition of $\tilde X^{k,m_0,m_1}$ is given in~\eqref{e.tilde_XY}).

The first step is to obtain a weighted $C^1$ estimate of $Q_s$.

\begin{lemma}\label{l:Qs_C1}
Let $m_1 > 3+ \gamma + 2s$ and $m_0$ sufficiently large depending only on $m_1$, $s$, and $\gamma$. The following inequality holds
\[
	\|Q_{\rm s}(g,f)\|_{L^{\infty,m_{1}}}
	\lesssim
	\|g\|_{L^{\infty,m_{1}}}
		\left(
			\|f\|_{L^{\infty,m_{0}}}
			+
			\|\nabla_{v}f\|_{L^{\infty,m_{1}}}
		\right).
\]
\end{lemma}
\begin{proof}
Let $\mu = 1$ if $\gamma \leq -1$ and $\mu = \frac{-\gamma - 2s}{1-2s}$ otherwise.  Fix $r = \vv^\mu/2$.  We first decompose the integral into two parts:
\[
	|Q_{{\rm s}}(g,f)\vv^{m_{1}}|
	=
	\int \left|\vv^{m_1}(f(v') - f(v)) K_g(v,v')\right| dv'
	\leq I_1 + I_2,
\]
		where $I_1$ and $I_2$ are the integrals over $B_r(v)$ 
		and $B_r(v)^c$, respectively.
Applying \Cref{l:new kernel} and using that if $\xi \in B_r(v)$ then $\langle \xi\rangle \approx \vv$, we bound $I_1$ as
\be\label{e.c681}
	\begin{split}
		I_1
		&\lesssim
		      \|\nabla_{v}f\|_{L^{\infty,m_1}}
		\int_{ B_{r}(v)}  |v-v'| K_{|g|}(v,v') dv'
		\lesssim
			\|\nabla_{v}f\|_{L^{\infty,m_1}}
			r^{1-2s} \int |g(w)||v-w|^{\gamma + 2s} dw.
	\end{split}
\ee
We are finished after bounding the integral by $\vv^{\gamma + 2s} \|g\|_{L^{\infty, m_0}}$ and using the definition of $r$.

The first step to handle $I_2$ is to split it into the parts containing $f(v)$ and $f(v')$ via the triangle inequality.  Call these integrals $I_{21}$ and $I_{22}$, respectively.  Using \Cref{l.is3.4} again, we see that
\[
	\begin{split}
		I_{21}
			= \int_{B_r(v)^c} \vv^{m_1} |f(v)| K_{|g|}(v,v') dv'
			\lesssim \|f\|_{L^{\infty,m_0}} \vv^{-m_0} r^{-2s} \int g(z) |v-z|^{\gamma + 2s} dv'.
	\end{split}
\]
Bounding the last integral using $\|g\|_{L^{\infty,m_1}}$ and using the definition of $r$ finishes the estimate of $I_{21}$.

The last integral, that of $I_{22}$ requires further decomposition into $I_{221}$ and $I_{222}$ over the domains $B_r(v)^c \cap B_{\vv/2}^c$ and $B_r(v)^c \cap B_{\vv/2}$.  The former is easy to handle using
	\[
		|f(v')|
			\leq \|f\|_{L^{\infty,m_0}} \vvp^{-m_0}
			\lesssim \|f\|_{L^{\infty,m_0}} \vv^{-m_0}
	\]
	where we used that $\vvp \gtrsim \vv$.  The rest of the bound follows exactly as for $I_{21}$.
	
	As for $I_{222}$, notice that for such $v'$, $|v-v'| \approx \vv$.  We use this, along with \Cref{l: w0524}, to find
\be\label{e.c682}
	\begin{split}
		I_{222}
			&\lesssim \|g\|_{L^{\infty,m_1}} \int_{B_r(v)^c \cap B_{\vv/2}} \frac{\vv^{m_1} f(v')}{\vv^{3+2s}}  \vv^{3+\gamma+2s-m_1} dv'\\
			&\lesssim \|g\|_{L^{\infty,m_1}} \int_{B_r(v)^c \cap B_{\vv/2}} \vv^\gamma f(v') dv'
			\lesssim \|g\|_{L^{\infty,m_1}} \|f\|_{L^{\infty,m_0}}.
	\end{split}
\ee
Combining this with the above estimates finishes the proof.
\end{proof}

Next we give the key estimate for constructing a solution.  To that end, we present a proposition that plays the role of \Cref{p:existence} above.  Recall the space $\tilde Y^{m_0,m_1}$ from~\eqref{e.tilde_XY}.
\begin{proposition} [Propagation of the weighted $C^1$ bounds]
	\label{l03}
	Fix any $m_1 > 3 + \gamma + 2s$ and $m_0$ sufficiently large depending only on $m_1$, $\gamma$ and $s$.  Suppose that $f_{\rm in} \in \tilde X^{m_0,m_1}$, and $g, R \in \tilde Y^{m_0,m_1}_T$.  If $f$ solves~\eqref{e.w316} then, there is a constant $C>0$ depending only on $m$, $s$, and $\gamma$ such that
	\[
		\|f\|_{\tilde Y^{m_0,m_1}}
			\lesssim \exp\Big\{C\int_{0}^{T} \|g(t)\|_{\tilde X^{m_0,m_1}}dt\Big\}
				\left( \|f_{\rm in}\|_{\tilde X^{m_0,m_1}}
					+ T \|R\|_{\tilde Y_T^{m_0,m_1}}\right).
	\]
\end{proposition}
\begin{proof}
	First notice that the proof of the bound
	\be\label{e.c683}
		\|f\|_{L^\infty([0,T];L^{\infty,m_0})}
			\lesssim e^{C \int_0^T\|g(t)\|_{L^{\infty,m_0}} dt } \Big(\|f_{\rm in}\|_{\tilde X^{m_0,m_1}}
					+ \int_0^T \|R(t)\|_{ \tilde X^{m_0,m_1}} dt\Big).
	\ee
	is exactly the same as the (brief) proof in \cite[Proposition~3.1]{HST_boltz_wp} and, hence, is omitted here.  We note that it is a simpler version of the proof of the bounds on the derivatives that follows.
	
We now focus instead on bounding $\nabla_{x,v}f$.  Fix $\phi(t)$ to be an increasing function to be determined such that $\phi(0) = \|\nabla_{x,v}f_{\rm in}\|_{L^{\infty,m_1}}$, and let $F(t,x,v) = \phi(t) \vv^{-m_1}$.  Clearly we have that
\be\label{e.c651}
	F(0,x,v) > \max\left\{|\partial_{x_i} f_{\rm in}(x,v)|, |\partial_{v_i}f(x,v)| : i \in \{1,2,3\}\right\}
		\quad\text{ for all } (x,v).
\ee
Let $t_0$ be the first time that the above inequality is violated.  If $t_0$ does not exist, we are finished.  Hence, we argue by contradiction assuming that there exists $t_0 \in [0,T]$.  Without loss of generality\footnote{Indeed, the only technical issue here is if the inequality is violated at $|v|=\infty$.  One may sidestep this by simply including a cutoff as a multiplicative factor of the initial data and of $R$.  It then follows from standard facts about the heat equation that $f$ and its derivatives decay as a Gaussian at high velocities.  The cutoff can be removed by a limiting procedure.}, we may assume that there exists $(x_0,v_0) \in \T^3 \times \R^3$ such that equality above holds in~\eqref{e.c651} at the point $(t_0,x_0,v_0)$.  Assume momentarily that
\be
	F(t_0,x_0,v_0)
		= \partial_{x_1} f(t_0,x_0,v_0).
\ee
The cases where $i = 2,3$ are clearly analogous, as are the case when a negative sign appears in the equality (i.e., $F = - \partial_{x_1} f$).  The case when the derivative is in the $v$ variable is slightly more complicated as new terms arise, but these new terms can be handled in a straightforward way.

Since $F - \partial_{x_1} f \geq 0$ on $[0,t_0]\times\T^3\times\R^3$, we find
\be\label{e.c652}
	0
		\geq \partial_t (F - \partial_{x_1} f)
			+ v\cdot \nabla_x(F - \partial_{x_1} f)
			- (\epsilon + (1-\sigma)) \Delta_{x,v} (F - \partial_{x_1} f)
			- \sigma Q_{\epsilon,\delta}(g, F - \partial_{x_1} f).
\ee
We use this to derive a contradiction.

On the one hand, an explicit computation for $F$, along with \cite[Proposition 3.1.(v)]{HST_boltz_wp} yields
\be\label{e.c653}
	\partial_t F + v\cdot\nabla_x F - (\epsilon + (1-\sigma)) \Delta_{x,v} F - \sigma Q_{\epsilon, \delta}(g,F)
		\geq \phi' \vvO^{-m_1}
			- C\phi (1 + \|g\|_{L^{\infty,m_1}}) \vvO^{-m_1},
\ee
where we used that $m_1 > 3 + \gamma + 2s$, a condition of the quoted result.

On the other hand, using \Cref{l:Qs_C1}, we find
\be\label{e.c654}
	\begin{split}
		\partial_t \partial_{x_1} f
			&+ v\cdot\nabla_x\partial_{x_1} f
			- (\epsilon + (1-\sigma)) \Delta_{x,v} \partial_{x_1} f 
			- \sigma Q_{\epsilon, \delta}(g,\partial_{x_1} f)
			= \sigma Q_{\epsilon, \delta}(\partial_{x_1} g, f) + \partial_{x_1}R
			\\& \lesssim \left(\|g(t_0)\|_{\tilde X^{m_0,m_1}} \left( \|f\|_{L^{\infty, m_0}} + \|\nabla f\|_{L^{\infty,m_1}}\right)
				+ \|R(t_0)\|_{\tilde X^{m_0,m_1}}\right) \vvO^{-m_1}\\
			&\leq \left(\|g(t_0)\|_{\tilde X^{m_0,m_1}} \left( \|f\|_{L^{\infty, m_0}} + \phi\right)
				+ \|R(t_0)\|_{\tilde X^{m_0,m_1}}\right) \vvO^{-m_1}.
	\end{split}
\ee
Using~\eqref{e.c683}, it is clear from~\eqref{e.c653} and~\eqref{e.c654} that we can choose $\phi$ to obtain a contradiction in~\eqref{e.c652}.  This yields a contradiction.  Hence~\eqref{e.c651} always holds, finishing the proof.
\end{proof}

As usual, once a priori estimates are established, the construction of a solution follows easily. In fact, in this case, the solution can be constructed exactly as in \cite{HST_boltz_wp}. Indeed, one can use the method of continuity as well as a smoothing argument in order to establish the existence of solutions to the linear problem.  After this, an iteration yields a solution to the nonlinear problem. As it is exactly the same as in \cite{HST_boltz_wp}, we omit the details.

One subtle issue that may cause worry is whether the process above provides a $W^{1,\infty}$ solution instead of $C^1$.   However, at the level of the method of continuity, the solutions constructed is smooth.  Hence all quantities $\partial_t f$, $\nabla_x f$, and $\nabla_v f$ are continuous and such continuity is passed through all (locally uniform) limits.

For uniqueness, one can actually simply use an $L^2$-based argument.  Indeed, a quick check of the arguments in \Cref{uniqueness} reveals that they can be adapted in a straightforward way to use only the $Y^{m_0,m_1}$ norms of two potential solutions $f$ and $g$.  Actually, the proof is {\em easier} in this case as there is no need to use the Sobolev embedding theorem.

The above concludes the proof of \Cref{t:C1_wellposed}.

\section*{Acknowledgments}
CH was partially supportd by NSF grant DMS-2003110. WW was partially supported by NSF grant DMS-1928930 while participating in the Mathematical Fluid Dynamics program hosted by the Mathematical Sciences Research Institute in Berkeley, California, during the Spring 2021 semester.

\bibliographystyle{abbrv}
\bibliography{boltz_wellposed_HW}

\end{document}